\documentclass{amsart}
\usepackage{amsmath,graphicx}

\usepackage{amsthm}
\usepackage{amssymb}
\usepackage{xypic}
\usepackage[square]{natbib}
\setcitestyle{numbers}
\theoremstyle{plain}
\newtheorem{theorem}{Theorem}
\newtheorem{proposition}{Proposition}
\newtheorem{corollary}{Corollary}
\newtheorem{lemma}{Lemma}
\theoremstyle{definition}
\newtheorem{definition}{Definition}

\newtheorem{example}{Example}

\def\shf{\mathcal}
\def\cshf{\mathfrak}

\def\id{\textrm{id }}

\DeclareMathOperator{\vspan}{span}
\title{Cosheaf Representations of Relations and Dowker Complexes}
\author{Michael Robinson}
\address{Mathematics and Statistics\\
American University\\
Washington, DC, USA}
\email{michaelr@american.edu}

\begin{document}
\maketitle

\begin{abstract}
The Dowker complex is an abstract simplicial complex that is constructed from a binary relation in a straightforward way.  Although there are two ways to perform this construction -- vertices for the complex are either the rows or the columns of the matrix representing the relation -- the two constructions are homotopy equivalent.  This article shows that the construction of a Dowker complex from a relation is a non-faithful covariant functor.  Furthermore, we show that this functor can be made faithful by enriching the construction into a cosheaf on the Dowker complex.  The cosheaf can be summarized by an integer weight function on the Dowker complex that is a complete isomorphism invariant for the relation.  The cosheaf representation of a relation actually embodies both Dowker complexes, and we construct a duality functor that exchanges the two complexes.  Finally, we explore a different cosheaf that detects the failure of the Dowker complex itself to be a faithful functor.
\end{abstract}

% TBD Choose what's a lemma, proposition, or theorem

\tableofcontents

\section{Introduction}

This article studies the structure of an abstract simplicial complex that is built according to a binary relation between two sets, as originally described by Dowker \cite{Dowker_1952}.  \emph{Dowker complexes}, as these simplicial complexes are now known, are simple to both construct and apply, finding use in many areas of mathematics and data science \cite{Ghrist_2014}.  Dowker's classic result is that there are two ways to build such an abstract simplicial complex, and that both of these complexes have the same homology.  This fact is a kind of \emph{duality}, because it arises from the transpose of the underlying relation's defining matrix.  It was later shown by Bj\"{o}rner \cite{bjorner1995topological} that the two dual Dowker complexes have homotopy equivalent geometric realizations.  

This article explains how the Dowker complex can be augmented with an integer weight function, and develops this idea into several functorial representations of the underlying relation.  Although the integer weight is not functorial, we show how it is the decategorification of a functorial, faithful cosheaf representation, and explore some of the implications of that fact.  In particular, Dowker's famous duality result arises as a functor that exchanges the base space and the space of global cosections of the cosheaf.  Considering only the Dowker complex without the weight function yields a non-faithful functor, since many different relations can have the same Dowker complex.  The article ends with the non-functorial construction of a cosheaf that captures the amount of \emph{redundancy} present in a relation -- a measure of how un-faithful the Dowker complex is for that particular relation.

Probably because of the topological nature of the duality result in \cite{Dowker_1952}, most of the literature discussing Dowker complexes focuses on their topological properties.  For instance, \cite{minian2010geometry} links the construction of the Dowker complex from a relation to the order complex of a partial order, and proves a number of homotopy equivalences.  Because Dowker complexes respect filtrations \cite{chowdhury2016persistent}, they seem ripe for use in topological data analysis, which typically focuses on the persistent homology of a filtered topological space.  This line of reasoning recently culminated in a functoriality result \cite[Thm. 3]{Chowdhury_2018}, which establishes that Dowker duality applies to the geometric realizations of sub-relations.  Specifically, consider a pair of nested subsets $R_1 \subseteq R_2 \subseteq (X \times Y)$ of the product of two sets $X$ and $Y$.  The Dowker complexes $D(X,Y,R_1)$ and $D(X,Y,R_2)$ for $R_1$ and $R_2$ and their duals $D(Y,X,R_1^T)$ and $D(Y,X,R_2^T)$ are related through a commutative diagram
\begin{equation*}
  \xymatrix{
    |D(X,Y,R_1)| \ar[r] \ar[d]_{\cong} & |D(X,Y,R_2)| \ar[d]_{\cong} \\
    |D(Y,X,R_1^T)| \ar[r] & |D(Y,X,R_2^T)| \\
    }
\end{equation*}
of continuous maps on their respective geometric realizations, in which the vertical maps are homotopy equivalences.

The paper \cite{Chowdhury_2018} appears to have set off a flurry of interest in Dowker complexes.  For instance, \cite{brun2019sparse} showed how to use Dowker complexes instead of \v{C}ech complexes for studying finite metric spaces.  Since topological data analysis often takes a finite metric space as an input, constructing a Vietoris-Rips complex is frequently an intermediate step; \cite{virk2019rips} shows how  Dowker complexes and Vietoris-Rips complexes are related.  Finally, \cite{salbu2019dowker} extended Dowker duality to simplicial sets, pointing the way to much greater generality.

The present paper is also inspired by the functoriality result \cite[Thm. 3]{Chowdhury_2018}, but in a somewhat different way.  Instead of focusing on sub-relations, we show that the Dowker complex construction is a functor from a category whose objects are relations and whose morphisms are relation-preserving transformations (Definition \ref{df:rel_cat}).  Furthermore, we show that the isomorphism classes of this category are completely characterized by two different \emph{weight functions} on the Dowker complex, and that these are derived from a faithful \emph{cosheaf representation} of the category.

Given a relation between two sets, the main results of this article are as follows:
\begin{enumerate}
\item The existence of two integer weighting functions, \emph{differential} and \emph{total} weights, on the Dowker complex for the relation that are complete isomorphism invariants (Theorems \ref{thm:differential_reconstruct} and \ref{thm:total_reconstruct}),
\item The Dowker complex is a functor from an appropriately constructed category of relations (Theorem \ref{thm:dowker_functor}),
\item The existence of faithful functors that render the relation into a cosheaf (Theorem \ref{thm:cosheaf_r0_functor} and Corollary \ref{cor:cosheaf_r_functor}) or sheaf (Theorem \ref{thm:sheaf_r0_functor}), whose (co)stalks determine the total weight function,
\item The space of global cosections of the cosheaf is the dual Dowker complex for the relation (Theorem \ref{thm:cosheaf_global_cosections}), and
\item There is a duality functor that exchanges the cosheaf's base space and space of global cosections (Theorem \ref{thm:cosheaf_dowker_duality}).
\end{enumerate}

\section{Recovery of a relation from a weight function on the Dowker complex}

\begin{definition}
  An \emph{abstract simplicial complex $X$ on a set $V_X$} consists of a set $X$ of subsets of $V_X$ such that if $\sigma \in X$ and $\tau \subseteq \sigma$, then $\tau \in X$.  Each $\sigma \in X$ is called a \emph{simplex of $X$}, and each element of $V_X$ is a \emph{vertex of $X$}.  Every subset $\tau$ of a simplex $\sigma$ is called \emph{face} of $\sigma$.
\end{definition}

It is usually tiresome to specify all of the simplices in a simplicial complex.  Instead, it is much more convenient to supply a \emph{generating set} $S$ of subsets of the vertex set.  The unique smallest simplicial complex containing the generating set is called the \emph{abstract simplicial complex generated by $S$}.

Let $R \subseteq X \times Y$ be a relation between finite sets $X$ and $Y$, which can be represented as a Boolean matrix $(r_{x,y})$.

\begin{definition}
  \label{df:dowker}
  The \emph{Dowker complex} $D(X,Y,R)$ is the abstract simplicial complex given by
  \begin{equation*}
    D(X,Y,R) = \{ [x_{i_0}, \dotsc, x_{i_k}] : \text{there exists a }y\in Y\text{ such that }(x_{i_j},y) \in R\text{ for all }j=0, \dotsc, k \}.
  \end{equation*}
  The \emph{total weight} is a function $t: D(X,Y,R)\to \mathbb{N}$ given by
  \begin{equation*}
    t(\sigma) = \# \{y \in Y : (x,y) \in R\text{ for all }x \in \sigma\}.
  \end{equation*}
  The \emph{differential weight} \cite{Ambrose_2020} is a function $d : D(X,Y,R) \to \mathbb{N}$ given by
  \begin{equation*}
    d(\sigma) = \# \{y \in Y: \left((x,y)\in R\text{ if } x\in \sigma\right) \text{ and } \left((x,y)\notin R\text{ if } x\notin \sigma\right) \}.
  \end{equation*}
\end{definition}

It is immediate by the definition that the Dowker complex is an abstract simplicial complex.

\begin{figure}
  \begin{center}
    \includegraphics[height=1in]{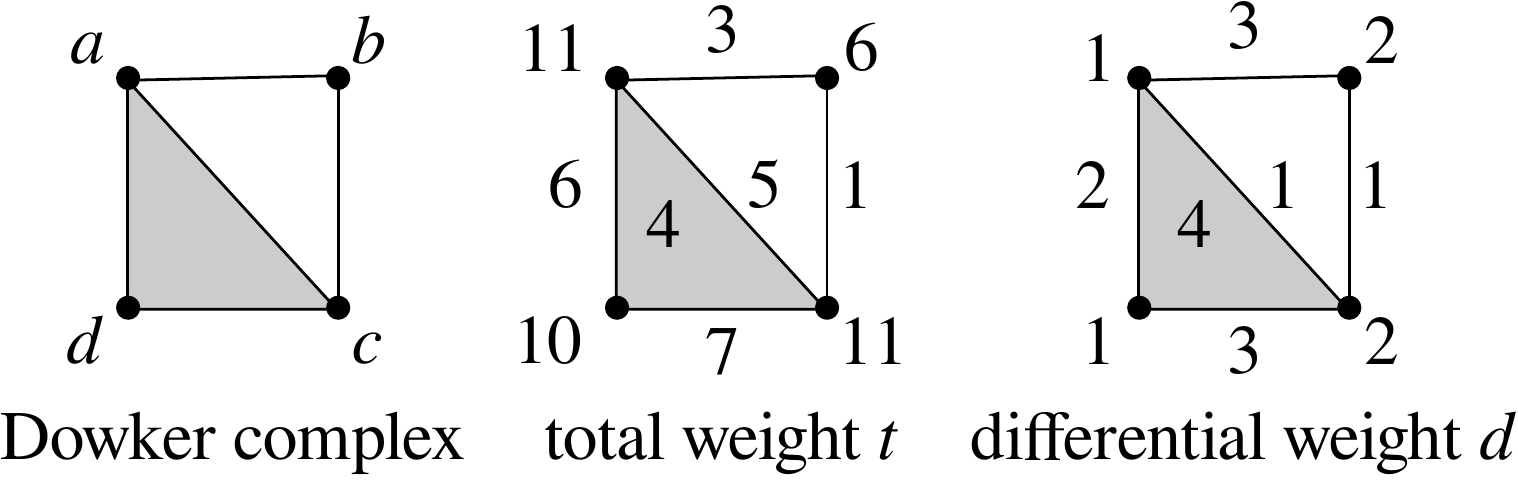}
    \caption{The Dowker complex for Example \ref{eg:eg1_dowker} (left), its total weight (center), and its differential weight (right).}
    \label{fig:eg1_dowker}
  \end{center}
\end{figure}

\begin{example}
  \label{eg:eg1_dowker}
  Consider the sets $X_1 = \{a,b,c,d\}$, and $Y_1=\{1,2, \dotsc, 20\}$ and the relation $R_1$ given by the matrix
  \begin{equation*}
    r_1 = \begin{pmatrix}
      1&0&0&0&0&0&1&1&{\bf 1}&1&1&1&{\bf 0}&0&0&0&1&1&1&{\bf 1}\\
      0&1&1&0&0&0&1&1&{\bf 1}&0&0&0&{\bf 1}&0&0&0&0&0&0&{\bf 0}\\
      0&0&0&1&1&0&0&0&{\bf 0}&1&0&0&{\bf 1}&1&1&1&1&1&1&{\bf 1}\\
      0&0&0&0&0&1&0&0&{\bf 0}&0&1&1&{\bf 0}&1&1&1&1&1&1&{\bf 1}\\
    \end{pmatrix}
  \end{equation*}  
  whose rows correspond to elements of $X_1$ and columns correspond to elements of $Y_1$.  The Dowker complex for this relation is generated by the simplices $[a,c,d]$, $[a,b]$, and $[b,c]$, a fact witnessed by the columns marked with {\bf bold} type.  The Dowker complex $D(X_1,Y_1,R_1)$ and its weighting functions are shown in Figure \ref{fig:eg1_dowker}.  Notice in particular that the differential weighting function counts the number of columns of $r_1$ of each simplex.  The total weighting accumulates all of the counts of columns for its faces as well.
\end{example}

\begin{figure}
  \begin{center}
    \includegraphics[height=1in]{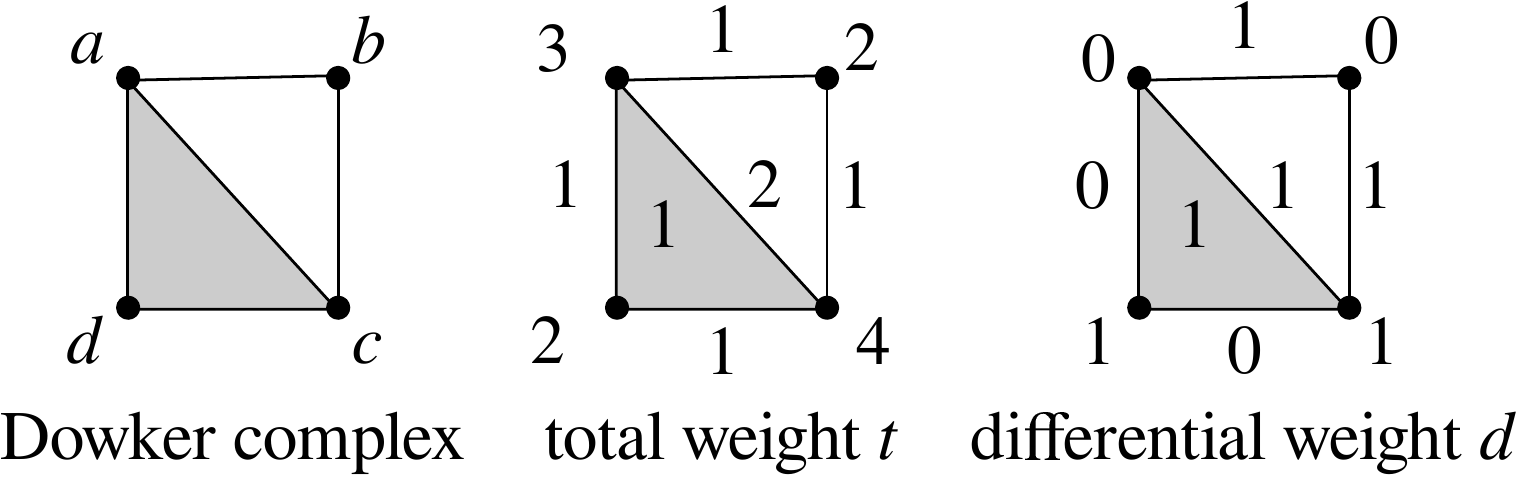}
    \caption{The Dowker complex for Example \ref{eg:eg2_dowker} (left), its total weight (center), and its differential weight (right).}
    \label{fig:eg2_dowker}
  \end{center}
\end{figure}

\begin{example}
  \label{eg:eg2_dowker}
  If we keep the same set $X_2=X_1$ as in Example \ref{eg:eg1_dowker}, but change the $Y_2$ set, with a different relation $R_2$ given by the matrix
  \begin{equation*}
    r_2 = \begin{pmatrix}
      1&0&1&0&0&1\\
      1&1&0&0&0&0\\
      0&1&1&1&0&1\\
      0&0&1&0&1&0\\
      \end{pmatrix}
  \end{equation*}
  we obtain the same Dowker complex, $D(X_2,Y_2,R_2) = D(X_1,Y_1,R_1)$.  However, as Figure \ref{fig:eg2_dowker} shows, the weight functions are different.
\end{example}

\begin{proposition}
  The sum of the differential weight $d$ on the Dowker complex $D(X,Y,R)$ is the number of elements of $Y$.
\end{proposition}

(Do not forget to count the differential weight of the empty simplex!)

\begin{proof}
  Observe that sets of the form $\{y \in Y: \left((x,y)\in R\text{ if } x\in \sigma\right) \text{ and } \left((x,y)\notin R\text{ if } x\notin \sigma\right) \}$ are disjoint for different $\sigma$ in $D(X,Y,R)$.  The sum of the differential weight is therefore the cardinality of 
    \begin{equation*}
      \sum_{\sigma \in D(X,Y,R)} d(\sigma) =
      \# \bigcup_{\sigma \in D(X,Y,R)} \{y \in Y: \left((x,y)\in R\text{ if } x\in \sigma\right) \text{ and } \left((x,y)\notin R\text{ if } x\notin \sigma\right) \},
    \end{equation*}
    which completes the argument.
\end{proof}

\begin{proposition}
  The total weight $t$ is a filtration on the Dowker complex $D(X,Y,R)$.
\end{proposition}
\begin{proof}
  This follows from showing that $t$ is order-reversing in the following way: if $\sigma \subseteq \tau$, then $t(\sigma) \ge t(\tau)$.

  Suppose $\sigma \subseteq \tau$, and that $y \in Y$ satisfies $(x,y) \in R$ for all $x \in \tau$.  Since $\sigma \subseteq \tau$, then $(x,y) \in R$ for all $x \in \sigma$ also.  Thus
  \begin{eqnarray*}
    \{y \in Y : (x,y) \in R\text{ for all }x \in \tau\} &\subseteq& \{y \in Y : (x,y) \in R\text{ for all }x \in \sigma\} \\
    \# \{y \in Y : (x,y) \in R\text{ for all }x \in \tau\} &\le& \# \{y \in Y : (x,y) \in R\text{ for all }x \in \sigma\} \\
    t(\tau) &\le& t(\sigma).
  \end{eqnarray*}
\end{proof}

\begin{theorem}
  \label{thm:differential_reconstruct}
  Given the Dowker complex $D(X,Y,R)$ and differential weight $d$, one can reconstruct $R$ up to a bijection on $Y$.
\end{theorem}
\begin{proof}
  The differential weight $d(\sigma)$ simply specifies the number of columns of the matrix $r$ for $R$ that can be realized as an indicator function for each $\sigma \in D(X,Y,R)$.  Thus, we can construct $r$ up to a column permutation.
\end{proof}

\begin{theorem}
  \label{thm:total_reconstruct}
  Given the Dowker complex $D(X,Y,R)$ and the total weight $t$, one can reconstruct $R$ up to a bijection on $Y$.
\end{theorem}
\begin{proof}
  We construct the relation matrix $r$ of $R$ iteratively.  Let $t_0 = t$.
  \begin{enumerate}
  \item Set $r_0$ to the zero matrix with no columns and as many rows as vertices of $D(X,Y,R)$.  That is, each row of $r_0$ corresponds to an element of $X$, so let us index rows of $r_0$ by elements of $X$.
  \item If $t_n(\sigma) = 0$ for all simplices $\sigma \in D(X,Y,R)$, declare $r= r_n$ and exit.
  \item Select a simplex $\sigma$ with $t_n(\sigma) \not= 0$ such that either there is no simplex $\tau$ containing $\sigma$ as a face, or if such a $\tau$ exists, then $t_n(\tau) = 0$.
  \item Define $r_{n+1}$ to be the horizontal concatenation of $r_n$ with $t_n(\sigma)$ columns, each an indicator function for $\sigma$.  That is, each new column is a Boolean vector $v$ given by
    \begin{equation*}
      v_x = \begin{cases}
        0 & \text{if }x\notin\sigma\\
        1 & \text{if }x\in\sigma.
      \end{cases}
    \end{equation*}
  \item Define a new function $t_{n+1}:D(X,Y,R) \to \mathbb{N}$ by
    \begin{equation*}
      t_{n+1}(\gamma)= \begin{cases}
        t_n(\gamma) - t_n(\sigma)&\text{if }\gamma \subseteq \sigma\\
        t_n(\gamma)&\text{otherwise.}
      \end{cases}
    \end{equation*}
  \item Increment $n$
  \item Go to step (3).
  \end{enumerate}

  Since $t_{n+1} < t_n$ on at least one simplex, and the relation $R$ is finite, the algorithm will always terminate.

  Secondly, the update step for $r_{n+1}$ by adding columns, establishes that $r$ relates the elements of $X$ contained in a given simplex by the appended $\sigma$ columns.

  Thirdly, notice that the apparent ambiguity in step (3) about selecting a simplex $\sigma$ merely results in a column permutation, since two maximal simplices do not interact with the update to $t_{n+1}$ in step (5), since another maximal simplex is not a face of $\sigma$.
\end{proof}

\begin{figure}
  \begin{center}
    \includegraphics[height=4in]{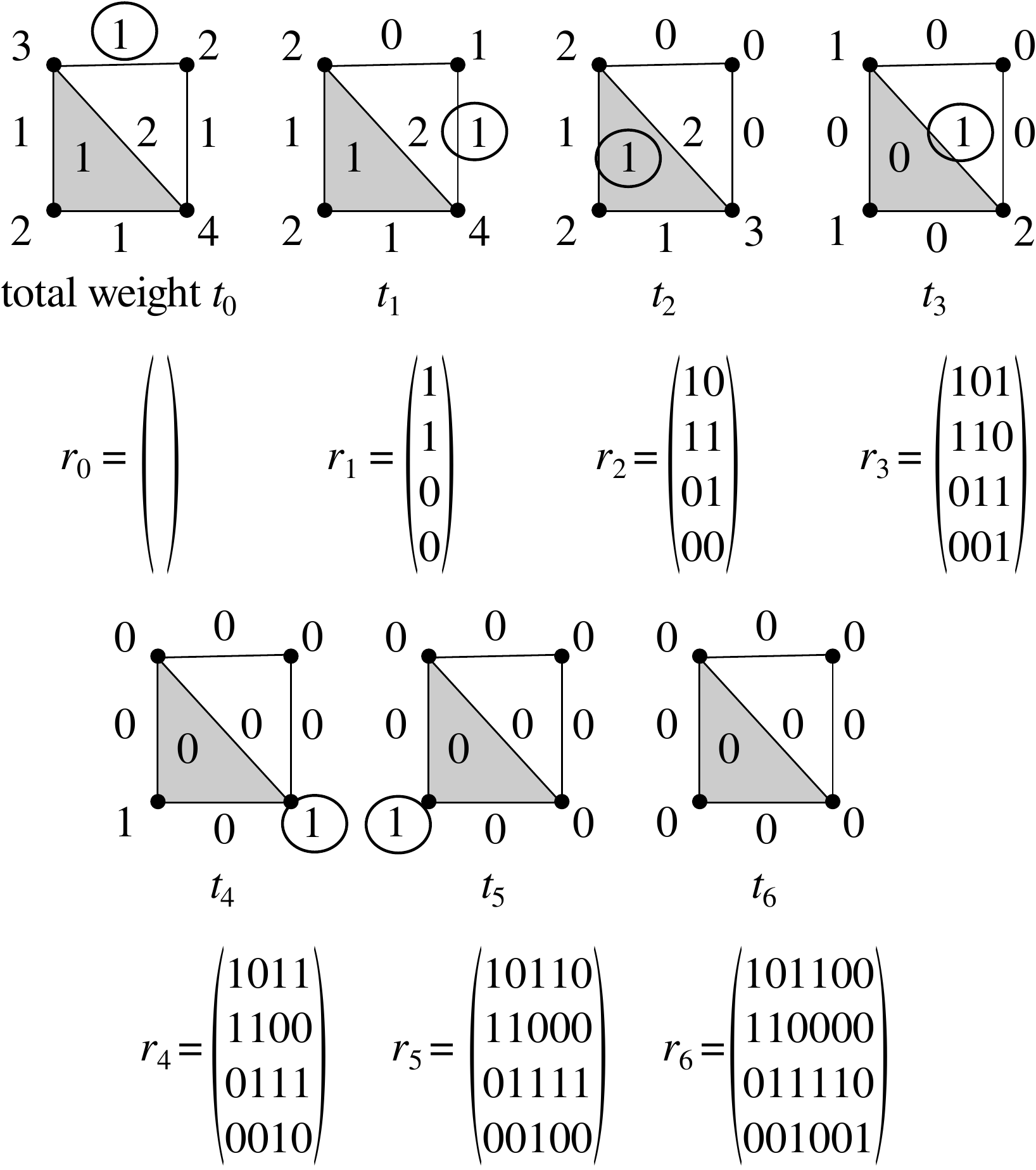}
    \caption{Recovering the relation from the total weight function as described in Example \ref{eg:eg3_dowker}.}
    \label{fig:eg3_dowker}
  \end{center}
\end{figure}

\begin{example}
  \label{eg:eg3_dowker}
  Starting with the relation from Example \ref{eg:eg2_dowker} and its total weight function, the algorithm described in the proof of Theorem \ref{thm:total_reconstruct} produces the relation matrix
  \begin{equation*}
    r= \begin{pmatrix}
      1&0&1&1&0&0\\
      1&1&0&0&0&0\\
      0&1&1&1&1&0\\
      0&0&1&0&0&1\\
      \end{pmatrix}
  \end{equation*}
  which differs from the original matrix $r_2$ by a cyclic permutation of the last three columns.  Figure \ref{fig:eg3_dowker} shows the progression of the steps of the algorithm.
\end{example}

\begin{figure}
  \begin{center}
    \includegraphics[height=1.5in]{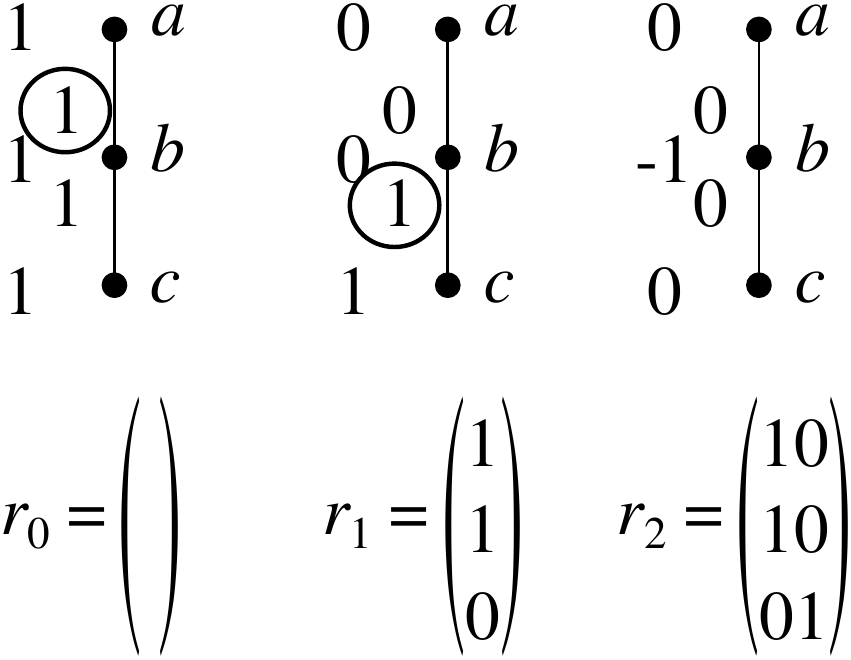}
    \caption{Attempting to recover the relation from a filtration that is not a total weight function can result in negative values, as described in Example \ref{eg:eg4_dowker}.}
    \label{fig:eg4_dowker}
  \end{center}
\end{figure}

\begin{example}
  \label{eg:eg4_dowker}
Not every nonnegative integer filtration of an abstract simplicial complex corresponds to the total weighting of a Dowker complex of a relation.  Although the algorithm in Theorem \ref{thm:total_reconstruct} may appear to run, it can produce negative values for the intermediate $t_\bullet$ weights, which cannot correspond to a number of columns in a relation!  For instance, the constant function on the simplicial complex generated by $[a,b]$ and $[b,c]$, as shown in Figure \ref{fig:eg4_dowker} is a filtration.  However, running the algorithm on this filtration produces a negative value at $[b]$, so we conclude that no relation can have this as a total weight function.
\end{example}

% TBD Your total and differential weight should be bijective by Mobious inversion in an incidence algebra on the poset of the ASC of the (multi)-hypergraph determined by your relation R: https://en.wikipedia.org/wiki/Incidence_algebra. Your total weight is the sum of the weights over each principle filter in the face order.

\section{Functoriality of the Dowker complex}

The Dowker complex $D(X,Y,R)$ is a functor between an appropriately constructed category of relations and the category of abstract simplicial complexes.  We prove this fact in Theorem \ref{thm:dowker_functor} along with a few other observations.

\begin{definition}
  \label{df:posets_cat}
  Consider an arbitrary set $P$ and a \emph{partial order} $\le$ on $P$.  A \emph{partial order} is a relation between elements in $P$ such that the following hold:
  \begin{enumerate}
  \item (Reflexivity) $x \le x$ for all $x \in P$,
  \item (Transitivity) if $x \le y$ and $y \le z$, then $x \le z$, and
  \item (Antisymmetry) if $x \le y$ and $y \le x$, then $x = y$.
  \end{enumerate}

  The \emph{category of partial orders} ${\bf Pos}$ has every partially ordered set $(P,\le)$ as an object.  Each morphism $g: (P,\le_P) \to (Q,\le_Q)$ consists of an \emph{order preserving function} $g: P \to Q$ such that if $x$ and $y$ are two elements of $P$ satisfying $x \le_P y$, then $g(x) \le_Q g(y)$ in $Q$.  Morphisms compose as functions on their respective sets.
\end{definition}

To see that ${\bf Pos}$ is a category, notice that the identity function is always order preserving and that the composition of two order preserving functions is another order preserving function.  Associativity follows from the associativity of function composition.

\begin{definition}
  The \emph{face partial order} for an abstract simplicial complex $X$ has the simplices of $X$ as its elements, and $\sigma \le \tau$ whenever $\sigma \subseteq \tau$.
\end{definition}

\begin{example}
  \label{eg:eg5_dowker}
  The face partial order for the Dowker complex shown in Figure \ref{fig:eg1_dowker} is given by its Hasse diagram
  \begin{equation*}
    \xymatrix{
                         &                      &  & [a,c,d]  & \\
      [a,b]              & [b,c]                & [a,c] \ar[ur]             & [a,d] \ar[u]     & [c,d] \ar[ul] \\
      [b] \ar[u] \ar[ur] & [a] \ar[ul] \ar[ur] \ar[urr] & [c]\ar[ul]\ar[u]\ar[urr] & [d] \ar[u]\ar[ur] &  \\
      }
  \end{equation*}
\end{example}

\begin{definition}
  \label{df:asc_cat}
Suppose $X$ and $Y$ are abstract simplicial complexes with vertex sets $V_X$ and $V_Y$, respectively.  A function $f: V_X \to V_Y$ on vertices is called a \emph{simplicial map} $f:X \to Y$ if it transforms each simplex $[v_0,\dotsb, v_k]$ of $X$ into a simplex $[f(v_0), \dotsb, f(v_k)]$ of $Y$, after removing duplicate vertices.  The category ${\bf Asc}$ has abstract simplicial complexes as its objects and simplicial maps as its morphisms.
\end{definition}

\begin{lemma}
  \label{lem:face_poset_functor}
  Let $f:X \to Y$ be a simplicial map.  For every pair of simplices $\sigma$, $\tau$ of $X$ satisfying $\sigma \subseteq \tau$, their images in $Y$ satisfy $f(\sigma) \subseteq f(\tau)$. 
\end{lemma}

\begin{proof}
  Since $f$ is a simplicial map, then $f(\sigma)$ is a simplex of $Y$ and so is $f(\tau)$.  If $\sigma \subseteq \tau$, then every vertex $v$ of $\sigma$ is also a vertex of $\tau$.  By the definition of simplicial maps, $f(v)$ is a vertex of both $f(\sigma)$ and $f(\tau)$.  Conversely, every vertex of $f(\sigma)$ is the image of some vertex $w$ of $\sigma$.
\end{proof}

\begin{proposition}
  \label{prop:face_functor}
  The face partial order is a covariant functor $Face: {\bf Asc}\to{\bf Pos}$.
\end{proposition}
\begin{proof}
  The construction of the face partial order from a simplicial complex establishes how the functor transforms objects.  Let us denote the face partial order for $X$ by $Face(X)$.  Lemma \ref{lem:face_poset_functor} establishes that every simplicial map $f:X \to Y$ induces an order preserving function $Face(f) : Face(X) \to Face(Y)$ on the face partial orders for $X$ and $Y$.  To verify that the functor is covariant, suppose that we have another simplicial map $g: Y \to Z$.  The composition of these is a simplicial map $g \circ f : X \to Z$ that induces an order preserving map $Face(g\circ f) : Face(X) \to Face(Z)$ on the face partial orders for $X$ and $Z$.  On the other hand, $Face(g) \circ Face(f) : Face(X) \to Face(Z)$ is also an order preserving map.  Any simplex $\sigma$ in $X$ can be reinterpreted as an element of $Face(X)$, which means that the simplex $(g \circ f)(\sigma)$ of $Z$ corresponds to the same element of $Face(Z)$ as does $g( f(\sigma))$, thought of as the image of an element of $Face(Y)$.
\end{proof}

\begin{definition}
  \label{df:rel_cat} (which has \cite[Sec 3.3]{JoyofCats} or \cite[pg. 54]{Rydeheard_1988} as a special case, and is manifestly the same as what appears in \cite{brun2019sparse})
  The \emph{category of relations} ${\bf Rel}$ has triples $(X,Y,R)$ for objects, in which $X$, $Y$ are sets and $R \subseteq X \times Y$ is a relation.  A morphism $(X,Y,R) \to (X',Y',R')$ in ${\bf Rel}$ is defined by a pair of functions $f:X\to X'$, $g:Y\to Y'$ such that $(f(x),g(y))\in  R'$ whenever $(x,y)\in R$.  Composition of morphisms is simply the composition of the corresponding pairs of functions, which means that ${\bf Rel}$ satisfies the axioms for a category.  It will be useful to consider the full subcategory ${\bf Rel_+}$ of ${\bf Rel}$ in which each object $(X,Y,R)$ has the property that for each $x \in X$, there is a $y \in Y$ such that $(x,y) \in R$, and conversely for each $y \in Y$, there is an $x \in X$ such that $(x,y) \in R$.
\end{definition}

\begin{example}
  \label{eg:rel_morph}
  Consider the relation $R_1$ between the sets $X_1=\{a,b,c,d,e\}$ and $Y_1=\{1,2,3,4,5\}$, given by the matrix
  \begin{equation*}
    r_1 = \begin{pmatrix}
      1&1&0&0&0\\
      1&0&1&0&0\\
      0&1&1&1&1\\
      0&0&1&1&0\\
      0&0&0&1&1\\
      \end{pmatrix}.
  \end{equation*}
  
  Suppose that $X_2=\{A,B,C\}$ and $Y_2=\{1,2,3,4,5\}$, that $f:X_1 \to X_2$ is given by
  \begin{equation*}
    f(a) = A,\;     f(b) = B,\;     f(c) = C,\;     f(d) = C,\;     f(e) = C,
  \end{equation*}
  and that $g: Y_1 \to Y_2$ is given by the identity function, namely
  \begin{equation*}
    g(1) = 1,\;    g(2) = 2,\;    g(3) = 3,\;    g(4) = 4,\;    g(5) = 5.
  \end{equation*}
  Then $(f,g)$ is a ${\bf Rel}$ morphism $(X_1,Y_1,R_1) \to (X_2,Y_2,R_2)$ if $R_2$ is given by the matrix
  \begin{equation*}
    r_2 =\begin{pmatrix}
    1&1&0&0&0\\
    1&0&1&0&0\\
    0&1&1&1&1\\
    \end{pmatrix}.
  \end{equation*}

  Additionally if $g': Y_1 \to Y_2$ is given by
  \begin{equation*}
    g'(1) = 1,\;    g'(2) = 2,\;    g'(3) = 3,\;    g'(4) = 3,\;    g'(5) = 3,
  \end{equation*}
  then $(f,g')$ is a ${\bf Rel}$ morphism $(X_1,Y_1,R_1) \to (X_2,Y_2,R_3)$ if $R_3$ is given by the matrix
  \begin{equation*}
    r_3 =\begin{pmatrix}
    1&1&0&0&0\\
    1&0&1&0&0\\
    0&1&1&0&0\\
    \end{pmatrix}.
  \end{equation*}
  However, $(f,g)$ is not a ${\bf Rel}$ morphism $(X_1,Y_1,R_1) \to (X_2,Y_2,R_3)$, since $(f(e),g(5))=(C,5)$ is not in the relation $R_3$ even though $(e,5)$ is in $R_1$.
\end{example}

\begin{theorem}
  \label{thm:dowker_functor}
  The Dowker complex defined in Definition \ref{df:dowker} is a covariant functor $D: {\bf Rel} \to {\bf Asc}$.
\end{theorem}
\begin{proof}
  Given the construction of the Dowker complex $D(X,Y,R)$ from $R \subseteq X \times Y$, we must show that
  \begin{enumerate}
  \item Each morphism in ${\bf Rel}$ translates into a simplicial map, and
  \item Composition of morphisms in ${\bf Rel}$ translates into composition of simplicial maps.
  \end{enumerate}
  To that end, suppose that $(X_1,Y_1,R_1)$, $(X_2,Y_2,R_2)$, and $(X_3,Y_3,R_3)$ are three objects in ${\bf Rel}$ with $f_1 : X_1 \to X_2$, $g_1 : Y_1 \to Y_2$ defining a morphism $(X_1,Y_1,R_1) \to (X_2, Y_2, R_2)$, and with $f_2 : X_2 \to X_3$, $g_2 : Y_2 \to Y_3$ defining a morphism $(X_2,Y_2,R_2) \to (X_3, Y_3, R_3)$.  The first claim to be proven is that $f_1$ is the vertex function for a simplicial map $D(X_1,Y_1,R_1) \to D(X_2,Y_2,R_2)$.  Suppose that $\sigma$ is a simplex of $D(X_1,Y_1,R_1)$.  Under the vertex map $f_1$, the set of vertices of $\sigma$ get transformed into the set
  \begin{equation*}
    f_1(\sigma) = \{f_1(x) : x \in \sigma \}.
  \end{equation*}
  But the defining characteristic of $\sigma$ is that there is a $y \in Y$ such that $(x,y) \in R_1$ for each $x \in \sigma$.  Using the function $g_1$ and the fact that the pair $(f_1,g_1)$ is a ${\bf Rel}$ morphism, we have that $(f_1(x),g_1(y)) \in R_2$ for every $x \in \sigma$.  This means that the set $f_1(\sigma)$ is actually a simplex of $D(X_2,Y_2,R_2)$.  Since $\sigma$ was arbitrary, this establishes that $f_1$ is a simplicial map.

  Since the composition of the ${\bf Rel}$ morphisms $(f_2,g_2) \circ (f_1,g_1)$ is defined to be $(f_2 \circ f_1, g_2 \circ g_1)$, this means that $D$ is a covariant functor, since this composition of ${\bf Rel}$ morphisms becomes a composition of simplicial maps.
\end{proof}

\begin{figure}
  \begin{center}
    \includegraphics[height=1.75in]{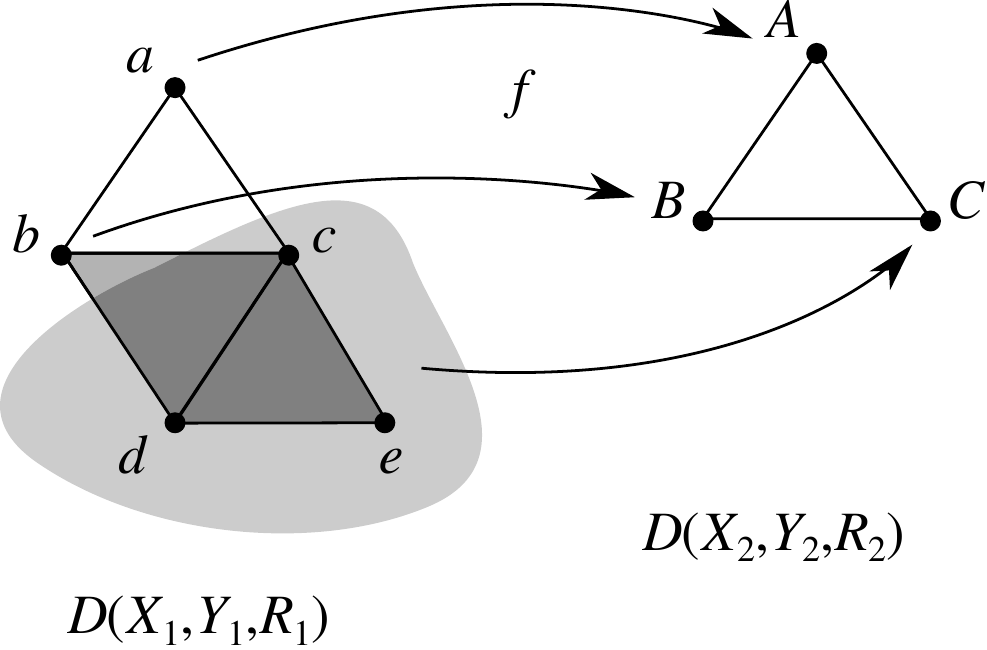}
    \caption{The simplicial map induced on the Dowker complexes by the ${\bf Rel}$ morphism $(f,g)$, which is used in Examples \ref{eg:rel_morph}, \ref{eg:dowker_map}, \ref{eg:rel_cosheaf_morph}, and \ref{eg:rel_cosheaf_dual_morph}.}
    \label{fig:dowker_map}
  \end{center}
\end{figure}

\begin{example}
  \label{eg:dowker_map}
  Continuing Example \ref{eg:rel_morph}, the Dowker complex $D(X_1,Y_1,R_1)$ is shown at left in Figure \ref{fig:dowker_map}.  The Dowker complexes $D(X_2,Y_2,R_2)$ and $D(X_2,Y_2,R_3)$ are identical, and are shown at right in Figure \ref{fig:dowker_map}.  The ${\bf Rel}$ morphism $(f,g)$ from Example \ref{eg:rel_morph} induces a simplicial map according Theorem \ref{thm:dowker_functor}.  The vertex function for this simplicial map is shown in Figure \ref{fig:dowker_map} as well.  The simplicial map collapses the simplex $[c,d,e]$ to the vertex $[C]$, while it collapses the simplex $[b,c,d]$ to the edge $[B,C]$.
\end{example}

Observe that ${\bf Pos}$ can be realized as a (non-full) subcategory of ${\bf Rel}$: each object in this subcategory is a partially ordered set $(X,\le_X)$ realized as $(X,X,\le_X)$, and each order preserving function $f: (X,\le_X) \to (Y,\le_Y)$ corresponds to a ${\bf Rel}$ morphism $(f,f) : (X,X,\le_X) \to (Y,Y,\le_Y)$ since the axioms coincide.  Beyond this relationship between ${\bf Pos}$ and ${\bf Rel}$, there is a different, functorial relationship.

\begin{proposition}
  \label{prop:posetrep_functor}
  There is a covariant functor $PosRep : {\bf Rel} \to {\bf Pos}$, called the \emph{poset representation} of a relation, that takes each $(X,Y,R)$ to a collection $PosRep(X,Y,R)$ of subsets of $2^X$, for which $A \in PosRep(X,Y,R)$ if there is a $y\in Y$ such that $(x,y)\in R$ for every $x\in A$.  The elements of $PosRep(X,Y,R)$ are ordered by subset inclusion.
\end{proposition}
\begin{proof}
  $PosRep$ translates morphisms in ${\bf Rel}$ into order preserving maps among partially ordered sets.  Specifically, the morphism $(X,Y,R) \to (X',Y',R')$ implemented by $f:X\to X'$ and $g:Y \to Y'$ is transformed into the function that takes $A\in PosRep(X,Y,R)$ to $f(A)$.  By definition there is a $y\in Y$ such that $(x,y)\in R$ for every $x\in A$.  Therefore, $(f(x),g(y))\in R'$ by construction. Since each $x' \in f(A)$ is given by $x'=f(x)$ for some $x$, this means that $(x',g(y))\in R'$ for all $x \in f(A)$.  Thus, $f(A) \in PosRep(X',Y',R')$.  Furthermore, the order relation among subsets is evidently preserved.
  
 The same argument from the proof of Theorem \ref{thm:dowker_functor} can be used \emph{mutatis mutandis} to show that $PosRep$ is a covariant functor, namely that composition of morphisms is preserved in order.
\end{proof}

\begin{proposition}
  \label{prop:posetrep_factor}
  The composition of the Dowker functor $D : {\bf Rel} \to {\bf Asc}$ with the face partial order functor $Face : {\bf Asc} \to {\bf Pos}$ yields the $PosRep$ functor.
\end{proposition}

In brief, the diagram
  \begin{equation*}
    \xymatrix{
      {\bf Rel} \ar[r]^{D} \ar[dr]_-{PosRep} & {\bf Asc} \ar[d]^{Face} \\
      & {\bf Pos}
      }
  \end{equation*}
  of functors commutes.

\begin{proof}
  To establish this result, we merely need to observe that the set of elements of $(Face \circ D)(X,Y,R)$ is the same set as $PosRep(X,Y,R)$, with the same order relation (subset inclusion).  Under that identification, the morphisms are the same, too.
\end{proof}

\begin{figure}
  \begin{center}
    \includegraphics[height=2.5in]{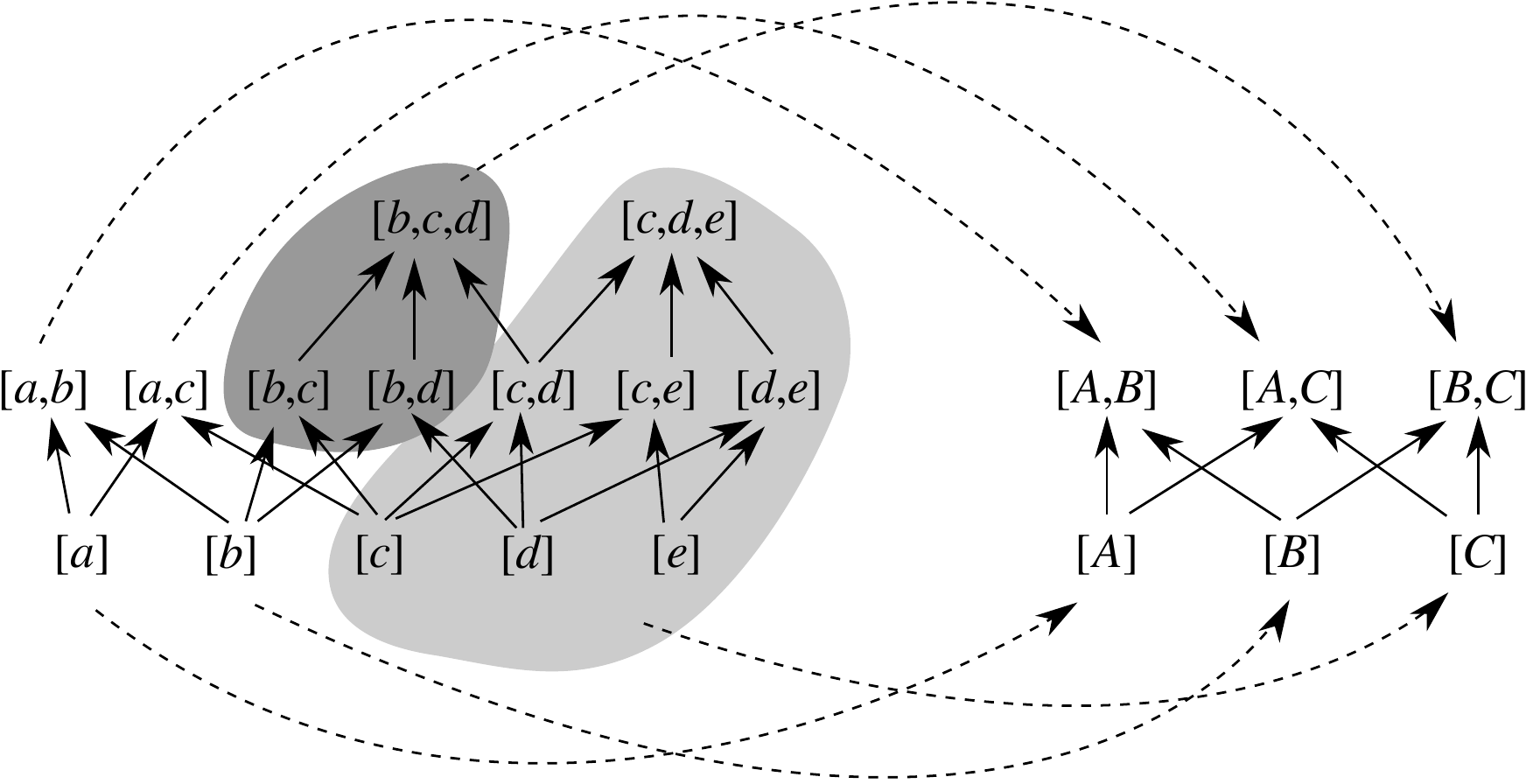}
    \caption{The order preserving map induced by the ${\bf Rel}$ morphism $(f,g)$ in Example \ref{eg:rel_morph} via the $PosRep$ functor.  See Example \ref{eg:dowker_poset_map}}
    \label{fig:dowker_poset_map}
  \end{center}
\end{figure}

\begin{example}
  \label{eg:dowker_poset_map}
  Continuing Examples \ref{eg:rel_morph} and \ref{eg:dowker_map}, the order preserving map induced by the ${\bf Rel}$ morphism $(f,g)$ is a bit tedious to construct from the data in Example \ref{eg:rel_morph}.  It is much more convenient to work from the simplicial map shown in Figure \ref{fig:dowker_map}.  Simply note that the only two nontrivial actions to be captured are related to the collapse of the simplices $[b,c,d]$ and $[c,d,e]$.  All of the faces of $[c,d,e]$ are mapped to $[C]$, while the remaining two edges of $[b,c,d]$ (those that are not also faces of $[c,d,e]$) are mapped to $[B,C]$.
\end{example}

\section{Functoriality of (co)sheaves on Dowker complexes}
\label{sec:functoriality}

  The Dowker functor $D: {\bf Rel} \to {\bf Asc}$ is not faithful: non-isomorphic relations can have the same Dowker complex.  The weighting functions distinguish between ${\bf Rel}$ isomorphism classes.  ${\bf Rel}$ morphisms sometimes induce transformations between weighting functions, for instance if the morphism transforms only the rows, or if the morphism permutes columns.  However, this does not always occur.  For instance, if the columns of one relation are included into another while simultaneously the rows are combined, the resulting transformation on the weighting functions is not described in a convenient way.  We really need a richer category for weighted Dowker complexes; this is found in the categories of cosheaves or of sheaves.  Specifically, each relation can be rendered as a cosheaf of sets, whose costalk cardinality is the total weight function on the Dowker complex.  Alternatively, a relation can be transformed into a sheaf of vector spaces, whose stalk dimensions specify a total weight function on the Dowker complex.  

  For convenience, let us begin by defining
  \begin{equation*}
    Y_\sigma = \{y\in Y : (x,y) \in R\text{ for all }x\in\sigma\}
  \end{equation*}
  for a simplex $\sigma$ of $D(X,Y,R)$.  The total weight function is simply the cardinality of this set: $t(\sigma) = \# Y_\sigma$.

  \begin{lemma}
    \label{lem:y_inclusion}
    If $\sigma \subseteq \tau$ are two simplices of $D(X,Y,R)$, then $Y_\tau \subseteq Y_\sigma$.
  \end{lemma}
  \begin{proof}
    Suppose $y \in Y_\tau$, so that $(x,y) \in R$ for all $x\in \tau$.  Since $\sigma \subseteq \tau$, it follows that $(x,y) \in R$ for all $x\in \sigma$.  Thus $y\in Y_\sigma$.
  \end{proof}

  \begin{lemma}
    \label{lem:y_morphism}
    For each simplex $\sigma$ of $D(X_1,Y_1,R_1)$, and each ${\bf Rel}$ morphism $(f,g) : (X_1,Y_1,R_1) \to (X_2,Y_2,R_2)$,
    \begin{equation*}
      g\left(\left(Y_1\right)_\sigma\right) \subseteq \left(Y_2\right)_{f(\sigma)}.
    \end{equation*}
  \end{lemma}
  \begin{proof}
    Suppose $z \in g\left(\left(Y_1\right)_\sigma\right)$, which means that $z = g(y)$ for some $y \in Y_1$ that satisfies $(x,y) \in R_1$ for all $x \in \sigma$.  Since $(f,g)$ is a ${\bf Rel}$ morphism, this means that $(f(x),g(y))=(f(x),z) \in R_2$ for all $x \in \sigma$ as well.  Therefore, $z \in \left(Y_2\right)_{f(\sigma)}$.
  \end{proof}

  \begin{corollary}
    \label{cor:y_diagram}
  As a result of Lemmas \ref{lem:y_inclusion} and \ref{lem:y_morphism}, the diagram
  \begin{equation*}
    \xymatrix{
      \left(Y_1\right)_\tau \ar[r]^-{g} \ar[d]_{\subseteq} & \left(Y_2\right)_{f(\tau)}\ar[d]^{\subseteq} \\
      \left(Y_1\right)_\sigma \ar[r]_-{g} & \left(Y_2\right)_{f(\sigma)}\\
      }
  \end{equation*}
  commutes when $(f,g)$ is a ${\bf Rel}$ morphism.
  \end{corollary}

  Although the total and differential weight functions are complete isomorphism invariants for ${\bf Rel}$, they are not functorial.  To remedy this deficiency, these weights can be thought of as summaries of a somewhat more sophisticated object: a \emph{cosheaf} or a \emph{sheaf}.

  \begin{definition} \cite{Baclawski_1975}
    \label{df:cosheaf_sheaf}
  A \emph{cosheaf of sets} $\cshf{C}$ on a partial order $(X,\le)$ consists of the following specification:
  \begin{itemize}
  \item For each $x \in X$, a set $\cshf{C}(x)$, called the \emph{costalk at $x$}, and
  \item For each $x \le y \in X$, a function $\left(\cshf{C}(x \le y)\right) : \cshf{C}(y) \to \cshf{C}(x)$, called the \emph{extension along $x \le y$}, such that
  \item Whenever $x \le y \le z \in X$, $\cshf{C}(x \le z) = \left(\cshf{C}(x \le y)\right) \circ \left(\cshf{C}(y \le z)\right)$.
  \end{itemize}
  Briefly, a cosheaf is a contravariant functor to the category ${\bf Set}$ from the \emph{category generated by $(X,\le)$}, whose objects are elements of $X$ and whose morphisms $x \to y$ correspond to ordered pairs $x \le y$.

  Dually, a \emph{sheaf of sets} $\shf{S}$ on a partial order $(X,\le)$ consists of the following specification:
  \begin{itemize}
  \item For each $x \in X$, a set $\shf{S}(x)$, called the \emph{stalk at $x$}, and
  \item For each $x \le y \in X$, a function $\left(\shf{S}(x \le y)\right) : \shf{S}(x) \to \shf{S}(y)$, called the \emph{restriction along $x \le y$}, such that
  \item Whenever $x \le y \le z \in X$, $\shf{S}(x \le z) = \left(\shf{S}(y \le z)\right) \circ \left(\shf{S}(x \le y)\right)$.
  \end{itemize}
  In this way, a sheaf is covariant functor to ${\bf Set}$ from the category generated by the partially ordered set $(X,\le)$.

  Given that every abstract simplicial complex $X$ corresponds to a partially ordered set $(X,\subseteq)$ via the face partial order $Face : {\bf Asc} \to {\bf Pos}$, we will often use \emph{(co)sheaves on an abstract simplicial complex} as a shorthand for (co)sheaves on the face partial order of an abstract simplicial complex.
  \end{definition}

  This definition of a (co)sheaf is traditionally that of a \emph{pre}(co)sheaf on a topological space.  The connection is that a (co)sheaf of sets of a partially ordered set is a minimal specification for a (co)sheaf on the partial order with the Alexandrov topology, via the process of \emph{(co)sheafification} \cite{Curry}.  Definition \ref{df:cosheaf_sheaf} is unambiguous in the context of this article, since we only consider (co)sheaves on partially ordered sets.

  \begin{definition}
    \label{df:coshvrep}
  We can use the information in a relation $(X,Y,R)$ to define a cosheaf $\cshf{R}^0$ on the face partial order\footnote{To keep the notation from becoming burdensome, we will abuse notation by regarding the abstract simplicial complex $D(X,Y,R)$ as a partially ordered set $(D(X,Y,R),\subseteq)$ rather than carrying around the $Face$ functor.} of $D(X,Y,R)$ by
  \begin{description}
  \item[Costalks] Each costalk is given by $\cshf{R}^0(\sigma) = Y_\sigma$, and
  \item[Extensions] If $\sigma \subseteq \tau$ in $D(X,Y,R)$, then the extension $\cshf{R}^0(\sigma \subseteq \tau): \cshf{R}^0(\tau) \to \shf{R}^0(\sigma)$ is the inclusion $Y_\tau \subseteq Y_\sigma$.
  \end{description}

  In a dual way, we can also define a sheaf $\shf{R}^0$ by:
\begin{description}
\item[Stalks] Each stalk is given by $\shf{R}^0(\sigma) = \vspan Y_\sigma$, and
\item[Restrictions] If $\sigma \subseteq \tau$ in $D(X,Y,R)$, then the restriction $\shf{R}^0(\sigma \subseteq \tau): \shf{R}^0(\sigma) \to \shf{R}^0(\tau)$ is defined to be the projection induced by the inclusion $Y_\tau \subseteq Y_\sigma$.
\end{description}
Notice that the basis for each stalk of $\shf{R}^0$ is the corresponding costalk of $\cshf{R}^0$.
  \end{definition}
  
\begin{corollary}
  For the cosheaf $\cshf{R}^0$ or sheaf $\shf{R}^0$ constructed from a relation $R$ as above,
  \begin{equation*}
    t(\sigma) = \# \cshf{R}^0(\sigma) = \dim \shf{R}^0(\sigma),
  \end{equation*}
  and
  \begin{equation*}
    d(\sigma) = \# \cshf{R}^0(\sigma) - \# \bigcup_{\sigma \subsetneqq \tau} \cshf{R}^0(\tau) = \dim \shf{R}^0(\sigma) - \dim \vspan_{\sigma \subsetneqq \tau} \shf{R}^0(\tau).
  \end{equation*}
\end{corollary}

The interpretation is that the total weight $t$ computes how many columns of $r$ start at $\sigma$, while the differential weight $d$ counts columns of $r$ that are related to the elements of $\sigma$ and no others.

  The main use of (co)sheaf theory is to formalize the notion of local and global consistency over some space, by way of identifying the data that are consistent with respect to the (co)sheaf.  These data are captured within \emph{global (co)sections}.
  
\begin{definition}
  For a cosheaf $\cshf{C}$ on a partially ordered set $(X,\le)$, consider the disjoint union of all costalks
  \begin{equation*}
    \bigsqcup_{x\in X} \cshf{C}(x).
  \end{equation*}
  Let $\sim$ be the equivalence relation on this disjoint union generated by $c_x \sim c_y$ whenever there exists an $x \le y$ in $X$ that satisfies
  \begin{enumerate}
  \item $c_x \in \cshf{C}(x)$ and
  \item $c_y \in \cshf{C}(y)$, such that
  \item $c_x = \left(\cshf{C}(x \le y)\right)(c_y)$.
  \end{enumerate}

  The \emph{set of global cosections} $\cshf{C}(X)$ of the cosheaf $\cshf{C}$ is given by the set of equivalence classes 
  \begin{equation*}
    \cshf{C}(X) = \left(\bigsqcup_{x\in X} \cshf{C}(x)\right) / \sim.
  \end{equation*}
  Each element of $\cshf{C}(X)$ is called a \emph{(global) cosection} of $\cshf{C}$.

  Dually, the \emph{set of global sections} of a sheaf $\shf{S}$ a on partially ordered set $(X,\le)$ is denoted by $\shf{S}(X)$ and is given by the subset
  \begin{equation*}
    \shf{S}(X) = \left\{s \in \prod_{x\in X} \shf{S}(x) : s_y = \left(\shf{S}(x \le y)\right)(s_x) \right\}.
  \end{equation*}
  Each element of $\shf{S}(X)$ is called a \emph{(global) section} of $\shf{S}$.
\end{definition}

\begin{example}
  \label{eg:eg2_cosheaf0}
  Recall the relation $R_2$ between $X=\{a,b,c,d\}$ and $Y=\{1,2,3,4,5,6\}$ from Example \ref{eg:eg2_dowker}, which was given by the matrix
    \begin{equation*}
    r_2 = \begin{pmatrix}
      1&0&1&0&0&1\\
      1&1&0&0&0&0\\
      0&1&1&1&0&1\\
      0&0&1&0&1&0\\
      \end{pmatrix}.
  \end{equation*}
  Using the partial order constructed for this relation in Example \ref{eg:eg5_dowker}, the cosheaf $\cshf{R}^0$ for the relation has the diagram
    \begin{equation*}
    \xymatrix{
                           &                      &  & *{\cshf{R}^0([a,c,d]])= \atop \{3\}} \ar[dr] \ar[d]\ar[dl]  & \\
      *{\cshf{R}^0([a,b])=\atop\{1\}} \ar[d] \ar[dr] & *{\cshf{R}^0([b,c])=\atop\{2\}} \ar[dl]\ar[dr] & *{\cshf{R}^0([a,c])=\atop\{3,6\}} \ar[dl] \ar[d]             & *{\cshf{R}^0([a,d])=\atop\{3\}} \ar[dll]\ar[d]     & *{\cshf{R}^0([c,d])=\atop\{3\}} \ar[dl]\ar[dll] \\
      *{\cshf{R}^0([b])=\atop\{1,2\}} & *{\cshf{R}^0([a])=\atop\{1,3,6\}} & *{\cshf{R}^0([c])=\atop\{2,3,4,6\}} & *{\cshf{R}^0([d])=\atop\{3,5\}}  &  \\
      }
    \end{equation*}
    where the numbers specify elements of $Y$ (also column indices of $r_2$).  The set of global cosections of this cosheaf is precisely
    \begin{equation*}
      \cshf{R}^0(X) = \{1,2,3,4,5,6\},
    \end{equation*}
    since the extension maps are all inclusions.  The equivalence classes involved merely identify equal elements of $Y$ (column indices) in the the disjoint union.
    Each global cosection of $\cshf{R}^0$ therefore corresponds to an element of $Y$ (equivalently, a column of $r_2$).

    The costalk cardinalities in the above diagram agree exactly with the total weight $t$ shown in Figure \ref{fig:eg2_dowker}.  Furthermore, the nonzero differential weights are
    \begin{eqnarray*}
      d([a,c,d]) &=& \# \cshf{R}^0([a,c,d]) = 1,\\
      d([a,b]) &=& \# \cshf{R}^0([a,b]) = 1,\\
      d([a,c]) &=& \#\cshf{R}^0([a,c]) - \#\cshf{R}^0([a,c,d]) = 2 - 1 = 1,\\
      d([c]) &=& \#\cshf{R}^0([c]) - \# \left(\cshf{R}^0([a,c,d]) \cup \cshf{R}^0([a,c]) \cup\cshf{R}^0([b,c]) \cup \cshf{R}^0([c,d]) \right) \\
      &=& 4 - \#\{2, 3,6\} = 4 - 3 = 1, \text{ and } \\
      d([d]) &=& \#\cshf{R}^0([c]) - \# \left(\cshf{R}^0([a,c,d]) \cup \cshf{R}^0([a,d]) \cup\cshf{R}^0([c,d]) \right) \\
      &=& 2 - \#\{ 3 \} = 2 - 1 = 1.
    \end{eqnarray*}
    
    By contrast, the sheaf is given by the diagram
    \begin{center}
      \includegraphics[height=2in]{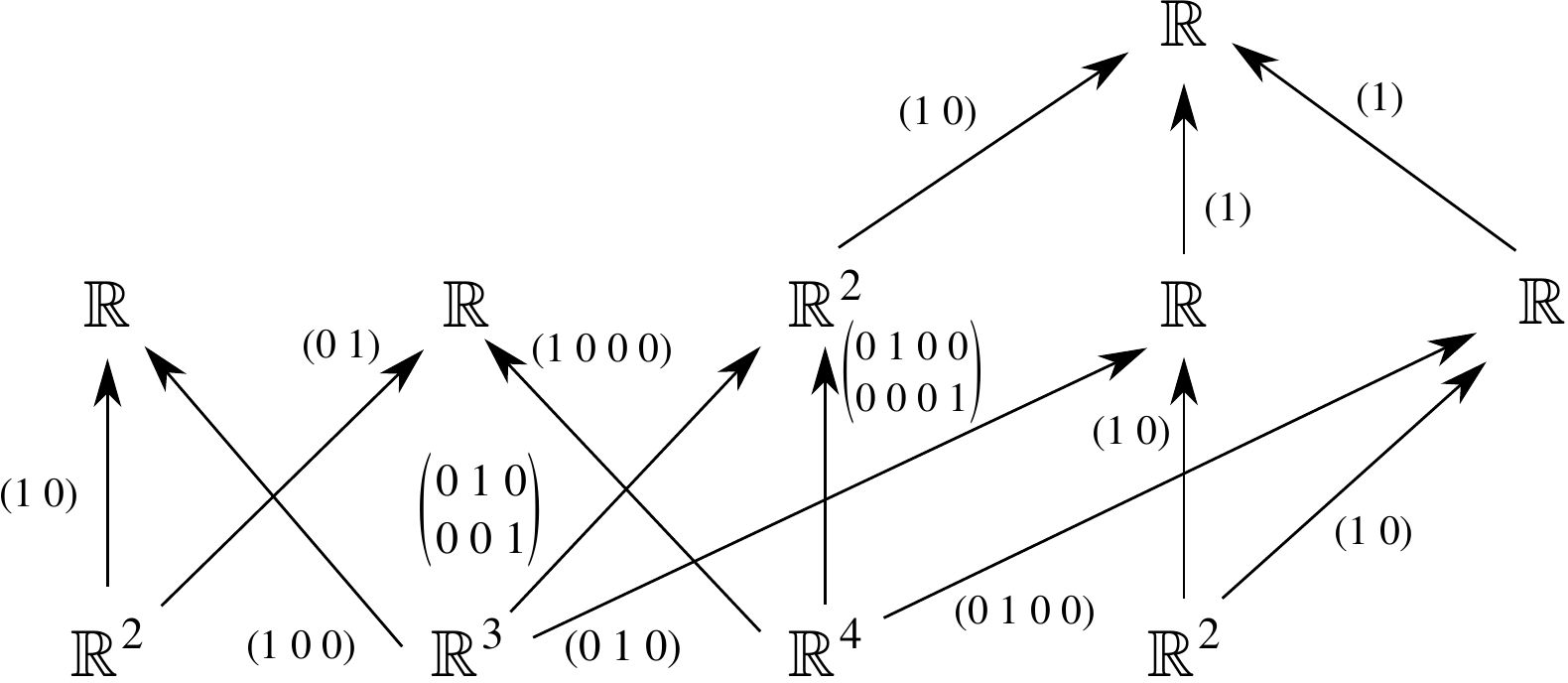}
    \end{center}
    We can interpret each global section of this sheaf as a formal linear combination of elements of $Y$, or a formal linear combination of columns of $r_2$.
\end{example}

To render cosheaves and sheaves into their own categories ${\bf CoShv}$ and ${\bf Shv}$, respectively, we need to define \emph{morphisms}.  Typical definitions (for instance, \cite[Def. 4.1.10]{Curry}) require the construction of morphisms between cosheaves or sheaves on the same partial order, but it is important to be a bit more general in our situation.  

\begin{definition} (\cite{Robinson_TSP_book} or \cite[Sec. I.4]{Bredon})
  \label{df:morphism}
  Suppose that $\cshf{R}$ is a cosheaf on a partially ordered set $(X,\le_X)$ and that $\cshf{S}$ is a cosheaf on a partially ordered set $(Y,\le_Y)$.  A \emph{cosheaf morphism} $m: \cshf{R} \to \cshf{S}$ along an order preserving \emph{base map}
  \begin{equation*}
    f: (X,\le_X) \to (Y,\le_Y)
  \end{equation*}
  consists of a set of \emph{component functions} $m_x: \cshf{R}(x) \to \cshf{S}(f(x))$ for each $x \in X$ such that the following diagram commutes
  \begin{equation*}
    \xymatrix{
      \cshf{R}(y) \ar[d]_{\cshf{R}(x \le_X y)} \ar[r]^-{m_y}& \cshf{S}(f(y))\ar[d]^{\cshf{S}(f(x) \le_Y f(y))}\\
      \cshf{R}(x) \ar[r]_-{m_x}& \cshf{S}(f(x))  \\
    }
  \end{equation*}
  
  Dually, suppose that $\shf{R}$ is a sheaf on a partially ordered set $(X,\le_X)$ and that $\shf{S}$ is a sheaf on a partially ordered set $(Y,\le_Y)$. 
  A \emph{sheaf morphism} $m: \shf{S}\to \shf{R}$ along an order preserving \emph{base map} $f: (X,\le_X) \to (Y,\le_Y)$ (careful: $m$ and $f$ \emph{go in opposite directions!}) consists of a set of \emph{component functions} $m_x: \shf{S}(f(x)) \to \shf{R}(x)$ for each $x \in X$ such that the following diagram commutes
  \begin{equation*}
    \xymatrix{
      \shf{S}(f(x)) \ar[r]^-{m_x} \ar[d]_{\shf{S}(f(x) \le_Y f(y))}&\shf{R}(x) \ar[d]^{\shf{R}(x \le_X y)}\\
      \shf{S}(f(y)) \ar[r]_-{m_y} &\shf{R}(y)\\
      }
  \end{equation*}
  for each $x \le_X y$.

  The \emph{category of cosheaves} ${\bf CoShv}$ (or \emph{category of sheaves} ${\bf Shv}$) consists of all cosheaves (or sheaves) on partially ordered sets as the class of objects, with cosheaf morphisms (or sheaf morphisms) as the class of morphisms.  Composition of morphisms in both cases is accomplished by simply composing the base map and component functions.
\end{definition}

\begin{lemma}
  \label{cor:base_functor}
  The transformation of a cosheaf to its underlying partial order is a covariant functor $Base : {\bf CoShv} \to {\bf Pos}$.  Likewise, the transformation of a sheaf to its underlying partial order is a contravariant functor $Base'$.
\end{lemma}
\begin{proof}
  Both of these statements follow immediately from the definition.
\end{proof}

\begin{lemma}
  \label{lem:morphisms_cosections}
  The transformation of cosheaves to global cosections is a covariant functor $\Gamma: {\bf CoShv} \to {\bf Set}$.  Specifically, every cosheaf morphism $p: \cshf{R} \to \cshf{S}$ induces a function on global cosections.  
\end{lemma}
\begin{proof}
  Suppose that $p: \cshf{R} \to \cshf{S}$ and $q: \cshf{S} \to \cshf{T}$ are cosheaf morphisms along order preserving base maps $f: (X,\le_X) \to (Y,\le_Y)$ and $g:(Y,\le_Y) \to (Z,\le_Z)$.  Let us use these data to define a function $p_* : \cshf{R}(X) \to \cshf{S}(Y)$ (and a function $q_* : \cshf{S}(Y) \to \cshf{T}(Z)$ by the same construction) such that $(q_* \circ p_*) = (q \circ p)_*$.  To that end, consider a cosection $c$ of $\cshf{R}$.  This is an element of
  \begin{equation*}
    \left(\bigsqcup_{x\in X} \cshf{R}(x)\right) / \sim.
  \end{equation*}
  Because of this, $c_x \in \cshf{R}(x)$ for some $x\in X$.  Define
  \begin{equation*}
    p_*(c) = p_x(c_x).
  \end{equation*}
  To verify that this is well-defined, suppose that $c_{x'} \in \cshf{R}(x')$ for some other $x' \in X$.  Under the equivalence relation $\sim$, the only way $c_{x'} \sim c_x$ can happen is if $x \le x'$ or $x' \le x$.  But since $c$ is a cosection, it happens that if $x' \le x$,
  \begin{equation*}
    \left(\cshf{R}(x' \le x)\right)(c_x) = c_{x'}.
  \end{equation*}
  Since $p$ is a cosheaf morphism, this means that
  \begin{eqnarray*}
    p_{x'}(c_{x'}) &=& \left(p_{x'} \circ \left(\cshf{R}(x' \le x)\right)\right)(c_x))\\
      &=& \left(\cshf{S}(f(x') \le f(x)) \circ p_x \right)(c_x),\\
      &=&  \left(\cshf{S}(f(x') \le f(x))\right) \left( p_x (c_x) \right),
  \end{eqnarray*}
  which implies that $p_{x'}(c_{x'}) \sim p_x(c_x)$ in $\cshf{S}(X)$.
  On the other hand, if $x \le x'$
  \begin{equation*}
    \left(\cshf{R}(x \le x')\right)(c_{x'}) = c_x.
  \end{equation*}
  Since $p$ is a cosheaf morphism, this means that
  \begin{eqnarray*}
    p_x(c_x) &=& \left(p_x \circ \left(\cshf{R}(x \le x')\right)\right)(c_{x'})\\
    &=& \left( \left(\cshf{S}(f(x) \le f(x'))\right) \circ p_{x'} \right)(c_{x'}),
  \end{eqnarray*}
  which also implies that $p_{x'}(c_{x'}) \sim p_x(c_x)$ in $\cshf{S}(X)$.  Thus, $p_*(c)$ is a well-defined global cosection of $\cshf{S}$.

  Repeating this construction with $q$, notice that
  \begin{eqnarray*}
    (q \circ p)_*(c) &=& (q \circ p)_x (c) \\
    &=&(q_x \circ p_x)(c)\\
    &=&(q_* \circ p_*)(c),
  \end{eqnarray*}
  which establishes covariance.
\end{proof}

\begin{theorem}
    \label{thm:cosheaf_r0_functor}
    The transformation $(X,Y,R) \mapsto \cshf{R}^0$ given in Definition \ref{df:coshvrep} is a covariant functor $CoShvRep^0: {\bf Rel} \to {\bf CoShv}$.  In particular, each ${\bf Rel}$ morphism induces a cosheaf morphism.  Furthermore, if the domain is restricted to ${\bf Rel_+}$, the functor becomes faithful.
\end{theorem}
  \begin{proof}
    Suppose that $(f_1,g_1) : (X_1,Y_1,R_1) \to (X_2,Y_2,R_2)$ and $(f_2,g_2) : (X_2,Y_2,R_2) \to (X_3,Y_3,R_3)$ are two ${\bf Rel}$ morphisms.  Suppose that $\cshf{R}_1^0$ is the cosheaf associated to $(X_1,Y_1,R_1)$ according to the recipe given in Definition \ref{df:coshvrep}, and likewise $\cshf{R}_2^0$ and  $\cshf{R}_3^0$ are the cosheaves associated to $(X_2,Y_2,R_2)$ and $(X_3,Y_3,R_3)$, respectively.  We first show how to construct a cosheaf morphism $m_1: \cshf{R}_1^0 \to \cshf{R}_2^0$.  

    Recognizing that the cosheaves $\cshf{R}_1^0$ and $\cshf{R}_2^0$ are written on the simplices of $D(X_1,Y_1,R_1)$ and $D(X_2,Y_2,R_2)$, recall that Theorem \ref{thm:dowker_functor} implies that $D(f_1)$ is a simplicial map $D(X_1,Y_1,R_1)\to D(X_2,Y_2,R_2)$, and that Propositions \ref{prop:posetrep_functor} and \ref{prop:posetrep_factor} imply that this can be interpreted as an order preserving function.  This is the order preserving base map along which $m_1$ is defined.

    Suppose that $\sigma \subseteq \tau$ in $D(X_1,Y_1,R_1)$.  As far as vertices are concerned, the diagram
    \begin{equation*}
      \xymatrix{
        \sigma \ar[d]_{\subseteq} \ar[r]^-{f} & f(\sigma) \ar[d]^{\subseteq}\\
        \tau \ar[r]_-{f} & f(\tau) \\
        }
    \end{equation*}
    commutes.  Corollary \ref{cor:y_diagram} therefore states that
  \begin{equation*}
    \xymatrix{
      \left(Y_1\right)_\tau \ar[r]^-{g} \ar[d]_{\subseteq} & \left(Y_2\right)_{f(\tau)}\ar[d]^{\subseteq} \\
      \left(Y_1\right)_\sigma \ar[r]_-{g} & \left(Y_2\right)_{f(\sigma)}\\
      }
  \end{equation*}
  commutes.  We therefore merely need to realize that according to Definition \ref{df:coshvrep}, this is equal to the diagram
  \begin{equation*}
    \xymatrix{
      \cshf{R}_1^0(\tau) \ar[r]^-{g} \ar[d]_{\cshf{R}_1^0(\sigma \subseteq \tau)} & \cshf{R}_2^0\left(f(\tau)\right)\ar[d]^{\cshf{R}_2^0(f(\sigma) \subseteq f(\tau))} \\
      \cshf{R}_1^0(\sigma) \ar[r]_-{g} & \cshf{R}_2^0\left(f(\sigma)\right)\\
      }
  \end{equation*}
  which establishes that $m_1$ is a cosheaf morphism, with $m_\sigma = g|_{(Y_1)_\sigma}$ as components.  Given that $m_2 : \cshf{R}_2^0 \to \cshf{R}_3^0$ can be constructed in the same way, the composition $(f_2,g_2)\circ(f_1,g_1)$ of ${\bf Rel}$ morphisms induces the composition of component functions, which is precisely the composition $m_2 \circ m_1$ of cosheaf morphisms.

  To show that this functor is faithful when restricted to objects in ${\bf Rel_+}$, a rather direct argument suffices.  Suppose that $(f,g)$ and $(f',g')$ are two ${\bf Rel}$ morphisms $(X_1,Y_1,R_1) \to (X_2,Y_2,R_2)$ in which $(X_1,Y_1,R_1)$ is an object of ${\bf Rel_+}$.  Recall that the latter constraint means that for every $x \in X_1$, there is a $y \in Y_1$ such that $(x,y) \in R_1$, and conversely for every $y \in Y_1$, there is an $x \in X_1$ such that $(x,y) \in R_1$.  To establish faithfulness, let us suppose additionally that $(f,g)$ and $(f',g')$ induce the same cosheaf morphism $m:\cshf{R}_1^0 \to \cshf{R}_2^0$.

  Let $y \in Y_1$ be given.  By assumption, there is an $x \in X_1$ such that $(x,y) \in R_1$, so there is also a simplex $\sigma$ (usually several simplices, actually) for which $y \in Y_\sigma$.  But, since both $(f,g)$ and $(f',g')$ both induce the same cosheaf morphism $m$, this means that
  \begin{equation*}
    g(y) = g|_{(Y_1)_\sigma}(y) = m_\sigma(y) = g'|_{(Y_1)_\sigma}(y) = g'(y).
  \end{equation*}
  Hence $g=g'$.

  Now let $x \in X_1$ be given.  By assumption, there is a $y \in Y_1$ such that $(x,y) \in R_1$, so this means that $[x]$ is a simplex of $D(X_1,Y_1,R_1)$.  Since both $(f,g)$ and $(f',g')$ induce the same cosheaf morphism $m$, this means that both $(f,g)$ and $(f',g')$ induce the same order preserving map on simplices of $D(X_1,Y_1,R_1) \to D(X_2,Y_2,R_2)$.  When restricted to vertices, this map is simply $f$ or $f'$, respectively, so they must also be equal.
  \end{proof}

  \begin{figure}
  \begin{center}
    \includegraphics[height=2.5in]{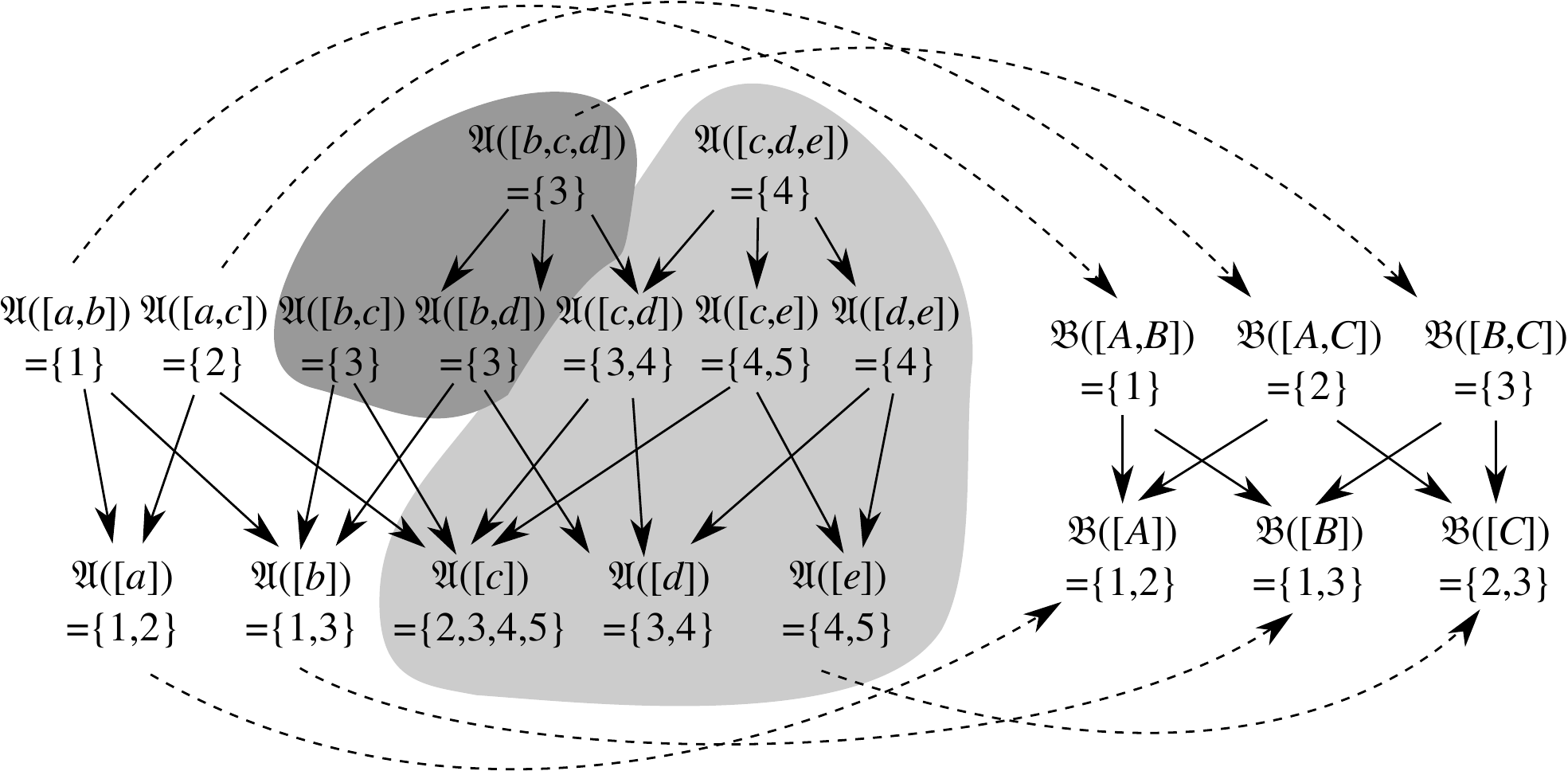}
    \caption{The cosheaf morphism induced by the ${\bf Rel}$ morphism $(f,g)$ in Example \ref{eg:rel_cosheaf_morph} via the $CoShvRep^0$ functor.}
    \label{fig:dowker_cosheaf_map}
  \end{center}
  \end{figure}
  
  \begin{example}
    \label{eg:rel_cosheaf_morph}
Consider again the relation $R_1$ between the sets $X_1=\{a,b,c,d,e\}$ and $Y_1=\{1,2,3,4,5\}$, given by the matrix
  \begin{equation*}
    r_1 = \begin{pmatrix}
      1&1&0&0&0\\
      1&0&1&0&0\\
      0&1&1&1&1\\
      0&0&1&1&0\\
      0&0&0&1&1\\
      \end{pmatrix},
  \end{equation*}
  from Example \ref{eg:rel_morph}.  However, this time let us consider a different morphism.  Define the relation $R_4$ between $X_4=\{A,B,C\}$ and $Y_4 =\{1,2,3\}$ given by
  \begin{equation*}
    r_4 = \begin{pmatrix}
      1&1&0\\
      1&0&1\\
      0&1&1\\
    \end{pmatrix}.
  \end{equation*}
  The function $f: X_1 \to X_4$ given by
  \begin{equation*}
    f(a) = A,\;     f(b) = B,\;     f(c) = C,\;     f(d) = C,\;     f(e) = C,
  \end{equation*}
  and the function $g: Y_1 \to Y_2$ given by
  \begin{equation*}
    g(1) = 1,\;    g(2) = 2,\;    g(3) = 3,\;    g(4) = 3,\;    g(5) = 3
  \end{equation*}  
  together define a ${\bf Rel}$ morphism $(X_1,Y_1,R_1) \to (X_4,Y_4,R_4)$.
  
  This relation morphism clearly maps each column of $r_1$ to a column of $r_4$, so it also acts on the costalks of the cosheaf representations.  If we define $\cshf{A} = CoShvRep^0(X_1,Y_1,R_1)$ and $\cshf{B} = CoShvRep^0(X_4,Y_4,R_4)$, the resulting cosheaf morphism $\cshf{A} \to \cshf{B}$ is given by the diagram shown in Figure \ref{fig:dowker_cosheaf_map}.  Notice that each dashed arrow represents a component map of the cosheaf morphism, and is given by restricting the domain of $g$ to each costalk, since this is how the columns are transformed. 
  \end{example}

\begin{theorem}
  \label{thm:sheaf_r0_functor}
  The transformation $R \mapsto \shf{R}^0$ given in Definition \ref{df:coshvrep} is a contravariant functor $ShvRep^0:{\bf Rel} \to {\bf Shv}$.  When restricted to ${\bf Rel_+} \to {\bf Shv}$, the functor becomes faithful.
\end{theorem}

The proof of this Theorem starts out exactly dual to that of the proof of Theorem \ref{thm:cosheaf_r0_functor}, but then diverges due to differences in the algebraic structure of the stalks.  The argument from that point looks different, but is actually the same (modulo a transpose, which is the duality) when restricted to basis elements of the stalk.

\begin{proof}
  Suppose that $(f_1,g_1) : (X_1,Y_1,R_1) \to (X_2,Y_2,R_2)$ and $(f_2,g_2) : (X_2,Y_2,R_2) \to (X_3,Y_3,R_3)$ are two ${\bf Rel}$ morphisms.  Suppose that $\shf{R}_1^0$ is the sheaf associated to $(X_1,Y_1,R_1)$ according to the recipe from Definition \ref{df:coshvrep}, and likewise $\shf{R}_2^0$ and  $\shf{R}_3^0$ are the sheaves associated to $(X_2,Y_2,R_2)$ and $(X_3,Y_3,R_3)$, respectively.  We first show how to construct a sheaf morphism $m_1: \shf{R}_2^0 \to \shf{R}_1^0$.  Given that $m_2 : \shf{R}_3^0 \to \shf{R}_2^0$ can be constructed in the same way, we show that the composition $(f_2,g_2)\circ(f_1,g_1)$ of ${\bf Rel}$ morphisms induces the composition $m_1 \circ m_2$ of sheaf morphisms.

  Recognizing that the sheaves $\shf{R}_1^0$ and $\shf{R}_2^0$ are written on the simplices of $D(X_1,Y_1,R_1)$ and $D(X_2,Y_2,R_2)$, recall that Theorem \ref{thm:dowker_functor} implies that $D(f_1)$ is a simplicial map $D(X_1,Y_1,R_1)\to D(X_2,Y_2,R_2)$, and that Propositions \ref{prop:posetrep_functor} and \ref{prop:posetrep_factor} imply that this can be interpreted as an order preserving function.  This is the order preserving map along which $m_1$ is defined.

  The component maps go the other way, and are expansions of the preimage of $g$.  For a simplex $\sigma$ of $D(X_1,Y_1,R_1)$, the $m_{1,\sigma} : \shf{R}_2^0(f(\sigma)) \to \shf{R}_1^0(\sigma)$ is given by the formula
  \begin{equation*}
    m_{1,\sigma}\left(\sum_{z \in (Y_2)_{f(\sigma)}} a_z z\right) = \sum_{z \in (Y_2)_{f(\sigma)},} \sum_{y \in g^{-1}(z)} a_z y.
  \end{equation*}
  To show that this is indeed a sheaf morphism requires showing that it commutes with the restriction maps.  This follows from the diagram of Corollary \ref{cor:y_diagram}, since that diagram explains how the basis vectors transform; the sheaf uses the dual of each map.  To show this explicitly, it suffices to show this for a pair of simplices $\sigma \subseteq \tau$ in $D(X_1,Y_1,R_1)$ and for a basis element $z \in (Y_2)_{f(\sigma)}$, because we can extend by linearity,
  \begin{eqnarray*}
    \left(\shf{R}_1^0(\sigma \subseteq \tau) \circ m_{1,\sigma}\right)(z) &=& \left(\shf{R}_1^0(\sigma \subseteq \tau)\right)\left(\sum_{y \in g^{-1}(z)} y \right) \\
    &=& \sum_{y \in g^{-1}(z)} \left(\shf{R}_1^0(\sigma \subseteq \tau)\right)\left( y \right) \\
    &=& \sum_{y \in g^{-1}(z)\text{ and }y\in (Y_1)_\tau} y.
  \end{eqnarray*}
  According to Lemma \ref{lem:y_morphism}, $y\in (Y_1)_\tau$ implies that $z \in (Y_2)_{f(\tau)}$.  Thus we may continue the calculation along the other path
  \begin{eqnarray*}
    \left(m_{1,\tau} \circ \shf{R}_2^0(f(\sigma) \subseteq f(\tau))\right)(z) &=& m_{1,\tau}(z)\\
    &=& \sum_{y \in g^{-1}(z)\text{ and }y\in (Y_1)_\tau} y,
  \end{eqnarray*}
  establishing commutativity of the diagram

  As for composition $(f_2,g_2)\circ(f_1,g_1)$ of ${\bf Rel}$ morphisms, suppose that $\sigma$ is a simplex of $D(X_1,Y_1,R_1)$.  We compute
  for $z \in (Y_3)_{f_2(f_1(\sigma))}$:
  \begin{eqnarray*}
    \left(m_{1,\sigma} \circ m_{2,f_1(\sigma)}\right) (z) &=& m_{1,\sigma} \left(\sum_{y \in g_2^{-1}(z)} y\right)\\
    &=& \sum_{w \in g_1^{-1}(y),} \sum_{y \in g_2^{-1}(z)} w\\
    &=& \sum_{w \in (g_2\circ g_1)^{-1}(z)} w,
  \end{eqnarray*}
  which is precisely what is induced by $(f_2 \circ f_1,g_2 \circ g_1)$.

  To show that this a faithful functor when restricted to ${\bf Rel_+}$, it suffices to recount the same argument for the cosheaf given in the proof of Theorem \ref{thm:cosheaf_r0_functor}, making the observation that the components of the cosheaf morphism are simply the functions on the basis elements of the stalks of the sheaf, after a transpose.
\end{proof}

Actually, the cosheaf $\cshf{R}^0$ seems a bit more natural than the sheaf $\shf{R}^0$!  At least, $\cshf{R}^0$ doesn't entrain any linear algebraic machinery, which may be ancillary to the main point.  On the other hand, the sheaf has cohomology, which may be worth exploring.

\begin{example}
  Notice that if we tried to define $\cshf{R}^0$ as a \emph{sheaf} of sets instead of a cosheaf -- using only the basis $Y_\sigma$ rather than its span -- then the proof of Theorem \ref{thm:sheaf_r0_functor} fails to work correctly, even if we reverse the partial order on $D(X,Y,R)$.  This is not an accident, since any functor ${\bf Rel} \to {\bf Shv}$ should compose with the global sections functor $\Gamma :{\bf Shv} \to {\bf Set}$ to ensure that ${\bf Rel}$ morphisms induce functions on the space of global sections.  This fails outright for a small example in which $X_1=\{A,B,C\}$, $X_2=\{A\}$, $Y_1=Y_2=\{a,b,c\}$, where $R_1$ and $R_2$ are given by the matrices
  \begin{equation*}
    r_1 = \begin{pmatrix}1&0&0\\0&1&1\end{pmatrix}\text{ and }r_2 = \begin{pmatrix}1&1&1\end{pmatrix}.
  \end{equation*}
  Noting that there is only one option to define $f: X_1 \to X_2$, we define $g = \id_{Y_1}$.  This is clearly a relation morphism as every pair $(x,y) \in R_1 \subseteq X_1 \times Y_1$ maps to a pair that are related by $R_2$.

  Using the reverse partial order, the sheaf diagram of basis elements for $(X_1,Y_1,R_1)$ is
  \begin{equation*}
    \xymatrix{
    \{a\} & \{b,c\} & \{c\} \\
    & \{c\} \ar[u] \ar[ur] & \\
    }
  \end{equation*}
  while the sheaf diagram for basis elements of $(X_2,Y_2,R_2)$ has only one element $\{a,b,c\}$.  (Both of these are identical to the cosheaf diagrams, since the unions in the Alexandrov topology are not shown.)  There is only one global section of the first sheaf, which consists of choosing $a$ for the leftmost simplex, and $b$ for the three elements on the right.  However, this cannot obviously be mapped to a global section of the second sheaf, since that needs to be a single element of $\{a,b,c\}$.  Conversely, if we consider a global section of the second sheaf as being any one of its elements, this cannot correspond to a global section of the first sheaf.  
\end{example}

\begin{corollary}
  The composition of $CoShvRep^0: {\bf Rel} \to {\bf CoShv}$ with the functor $Base: {\bf CoShv} \to {\bf Pos}$ that forgets the structure of the costalks is $PosRep : {\bf Rel} \to {\bf Pos}$.  This also works for the composition $(Base' \circ ShvRep^0): {\bf Rel} \to {\bf Shv} \to {\bf Pos}$.  Briefly,
  \begin{equation*}
    PosRep = Base \circ CoShvRep^0 = Base' \circ ShvRep^0.
  \end{equation*}
\end{corollary}
Notice that in the case of sheaves on partial orders, both functors are contravariant!

\section{Duality of cosheaf representations of relations}
\label{sec:duality}

The most striking fact proven in Dowker's original paper \cite{Dowker_1952} is that the homology of the Dowker complex is the same whether it is produced by the relation or by its transpose.  Dowker provides a direct, if elaborate, construction of a map inducing isomorphisms on homology.  This construction was later enhanced to a homotopy equivalence by \cite{bjorner1995topological}.  More recently, \cite{Chowdhury_2018} showed that the homotopy equivalence between these two complexes is functorial in a particular way.  This section shows that the duality is also visible in a somewhat different way: one Dowker complex is the base space of a particular cosheaf, while the other is its space of global cosections.

Let us begin by connecting the relation to its transpose.

\begin{definition}
  If $(X,Y,R)$ is a relation, then its \emph{transpose} is a relation $(Y,X,R^T)$ given by $(y,x) \in R^T$ if and only if $(x,y) \in R$.
\end{definition}

Evidently, the matrix for the transpose of a relation is simply the transpose of the original matrix.

\begin{lemma}
  The transformation $(X,Y,R) \mapsto (Y,X,R^T)$ defines a fully faithful covariant functor $Transp:{\bf Rel} \to {\bf Rel}$.
\end{lemma}
\begin{proof}
  Every ${\bf Rel}$ morphism $(f,g) : (X,Y,R) \to (X',Y',R')$ is transformed to $(g,f) : (Y,X,R^T) \to (Y',X',(R')^T)$.  Composition is still composition of functions and is preserved in order.
\end{proof}

  \begin{figure}
  \begin{center}
    \includegraphics[height=1.5in]{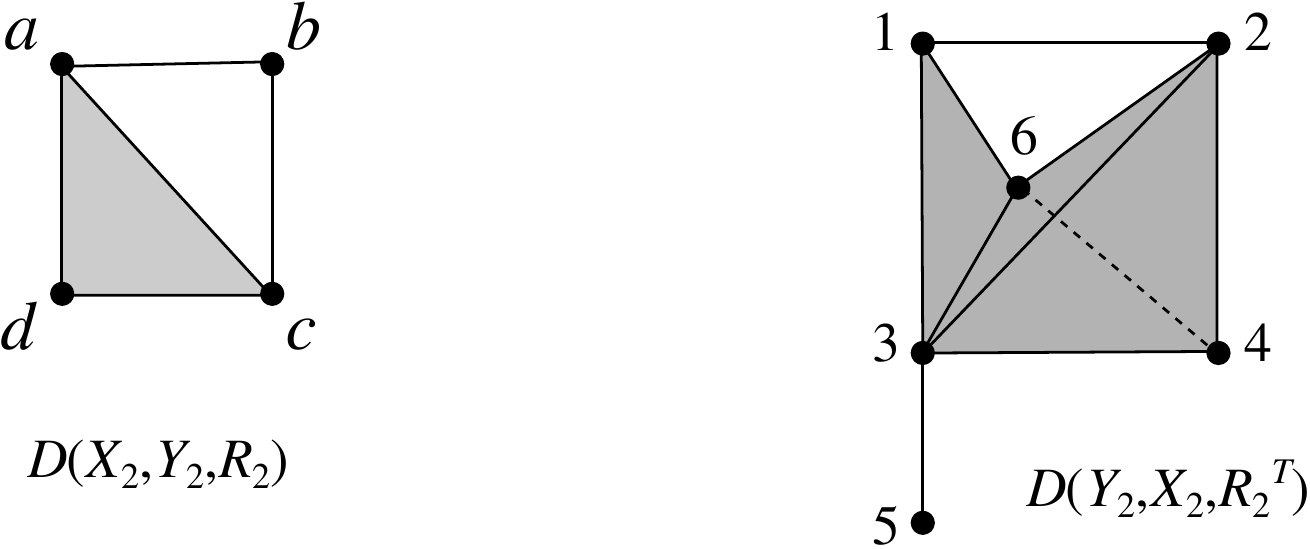}
    \caption{The Dowker complex for a relation $R_2$ and its transpose given in Example \ref{eg:eg2_dowker_transpose}.}
    \label{fig:eg2_dowker_transpose}
  \end{center}
  \end{figure}

  \begin{example}
    \label{eg:eg2_dowker_transpose}

  Recall the relation $R_2$ between $X_2=\{a,b,c,d\}$ and $Y_2=\{1,2,3,4,5,6\}$ from Example \ref{eg:eg2_dowker}, which was given by the matrix
  \begin{equation*}
    r_2 = \begin{pmatrix}
      1&0&1&0&0&1\\
      1&1&0&0&0&0\\
      0&1&1&1&0&1\\
      0&0&1&0&1&0\\
    \end{pmatrix}.
  \end{equation*}
  The transpose of this relation has the matrix
  \begin{equation*}
    r_2^T = \begin{pmatrix}
      1&1&0&0\\
      0&1&1&0\\
      1&0&1&1\\
      0&0&1&0\\
      0&0&0&1\\
      1&0&1&0\\
    \end{pmatrix}.
  \end{equation*}
  Their Dowker complexes are shown in Figure \ref{fig:eg2_dowker_transpose}.  Clearly these complexes have the same homotopy type!
\end{example}

\begin{definition}
  The category ${\bf CoShvAsc}$ consists of the full subcategory of ${\bf CoShv}$ whose objects are cosheaves on abstract simplicial complexes of abstract simplicial complexes, and whose extensions are simplicial inclusions.  Briefly, an object of ${\bf CoShvAsc}$ is a contravariant functor $\cshf{C}$ from the face partial order of an abstract simplicial complex to ${\bf Asc}$, with the additional condition that each extension $\cshf{C}(\sigma \subseteq \tau) : \cshf{C}(\tau) \to\cshf{C}(\sigma)$ is a simplicial map whose vertex function is an inclusion.
\end{definition}

\begin{definition}
  \label{df:cshf_r}
 The \emph{cosheaf representation of a relation $(X,Y,R)$} is a cosheaf $\cshf{R}=CoShvRep(X,Y,R)$ of abstract simplicial complexes, defined by the following recipe:
\begin{description}
\item[Costalks] If $\sigma$ is a simplex of $D(X,Y,R)$, then $\cshf{R}(\sigma) = D\left(Y_\sigma,\sigma,(R|_{\sigma,Y_\sigma})^T\right)$,
\item[Extensions] If $\sigma \subseteq \tau$ are two simplices of $D(X,Y,R)$, then the extension $\cshf{R}(\sigma\subseteq \tau) : \cshf{R}(\tau) \to \cshf{R}(\sigma)$ is the simplicial map along the inclusion $Y_\tau \hookrightarrow Y_\sigma$.
\end{description}
\end{definition}

The cosheaf $\cshf{R}^0 = CoShvRep^0(X,Y,R)$ for a relation $(X,Y,R)$ defined in Section \ref{sec:functoriality} is a sub-cosheaf of $\cshf{R}=CoShvRep(X,Y,R)$.  Evidently, $\cshf{R}$ is an object of ${\bf CoShvAsc}$.

\begin{figure}
  \begin{center}
    \includegraphics[height=2in]{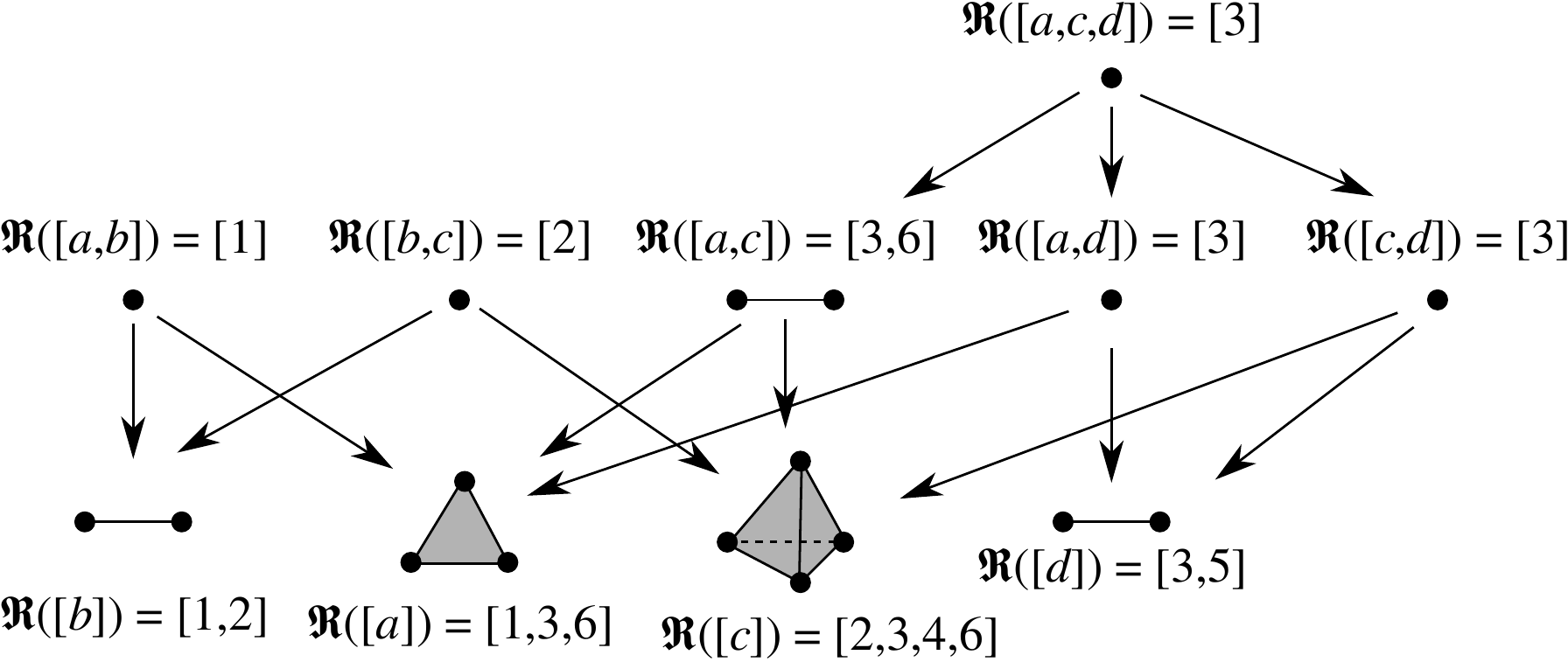}
    \caption{The diagram of the cosheaf $\cshf{R}$ defined in Example \ref{eg:eg2_cosheaf} for the relation $(X_2,Y_2,R_2)$ defined in Example \ref{eg:eg2_dowker}.}
    \label{fig:eg2_cosheaf}
  \end{center}
\end{figure}

\begin{figure}
  \begin{center}
    \includegraphics[height=1.125in]{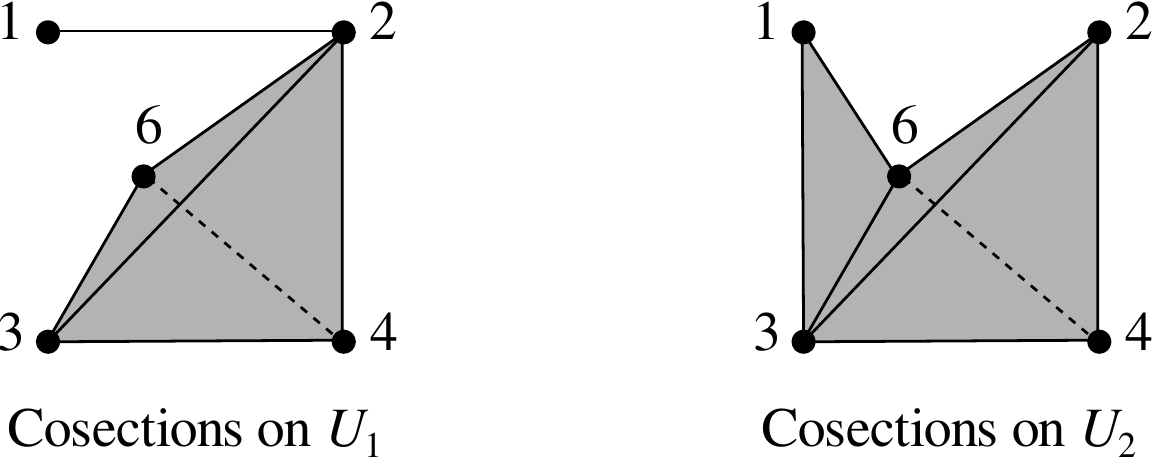}
    \caption{Some spaces of cosections of the cosheaf $\cshf{R}$ defined in Example \ref{eg:eg2_cosheaf}: (left) cosections on the set $U_1$, (right) cosections on the set $U_2$.}
    \label{fig:eg2_cosheaf_cosections}
  \end{center}
\end{figure}

\begin{example}
  \label{eg:eg2_cosheaf}
  Recall the relation $R_2$ between $X_2=\{a,b,c,d\}$ and $Y_2=\{1,2,3,4,5,6\}$ from Example \ref{eg:eg2_dowker}, which was given by the matrix
  \begin{equation*}
    r_2 = \begin{pmatrix}
      1&0&1&0&0&1\\
      1&1&0&0&0&0\\
      0&1&1&1&0&1\\
      0&0&1&0&1&0\\
    \end{pmatrix}.
  \end{equation*}
  The cosheaf $\cshf{R}^0 = CoShvRep^0(X_2,Y_2,R_2)$ was described in Example \ref{eg:eg2_cosheaf0}.  The diagram for $\cshf{R} = CoShvRep(X_2,Y_2,R_2)$ is shown in Figure \ref{fig:eg2_cosheaf}, which incorporates all of the data from the previous examples into a single figure.  Notice that each costalk shown in the diagram is a complete simplex.  The space of cosections over the set
  \begin{equation*}
    U_1 = \{[b],[c],[a,b], [a,c], [b,c], [c,d], [a,c,d]\}
  \end{equation*}
  is an abstract simplicial complex on the vertex set that is the union
  \begin{equation*}
    \cshf{R}([b]) \cup \cshf{R}([c]) = \{1,2\}\cup \{2,3,4,6\} = \{1,2,3,4,6\},
  \end{equation*}
  but is not the complete simplex on those vertices.  Instead, it is the simplicial complex shown at left in Figure \ref{fig:eg2_cosheaf_cosections}.  Likewise the space of cosections over the set
  \begin{equation*}
    U_2 = \{[a],[c],[a,b],[b,c],[a,c],[a,d],[c,d],[a,c,d]\}
  \end{equation*}
  is shown at right in Figure \ref{fig:eg2_cosheaf_cosections}.  From these two examples, it is clear that the space of global cosections is indeed $D(Y_2,X_2,R_2^T)$, as shown in Figure \ref{fig:eg2_dowker_transpose}.
\end{example}

\begin{lemma}
  \label{lem:costalks_complete_simplex}
For any simplex $\tau$ of $D(X,Y,R)$, the costalk $\cshf{R}(\tau) = D\left(Y_\tau,\tau,(R|_{\tau,Y_\tau})^T\right)$ is always a complete simplex on the vertex set $Y_\tau$.
\end{lemma}
\begin{proof}
  Every subset $\{y_0,y_1, \dotsc, y_n\}$ consisting of elements $y_i$ of $Y_\tau$ is a simplex of $D\left(Y_\tau,\tau,(R|_{\tau,Y_\tau})^T\right)$, since that merely requires there to be at least one $x \in \tau$ to exist such that $(x,y_i) \in R$ for all $i$.
\end{proof}

\begin{lemma}
  The extensions defined for the cosheaf $\cshf{R}$ for a relation $(X,Y,R)$ in Definition \ref{df:cshf_r} are well-defined simplicial maps.
\end{lemma}
\begin{proof}
  Suppose that $\sigma \subseteq \tau$ are two simplices of $D(X,Y,R)$. Lemma \ref{lem:costalks_complete_simplex} establishes that both $\cshf{R}(\sigma)$ and $\cshf{R}(\tau)$ are complete simplices.  Accordingly, consider the subset $\{y_0,y_1, \dotsc, y_n\}$ consisting of elements $y_i$ of $Y_\tau$.  Notice that by the definition of $Y_\tau$, for every $x \in \tau$ and every $i$, it follows that $(x,y_i) \in R$. Therefore, since $\sigma \subseteq \tau$, this condition also holds for every $x \in \sigma$.  Thus, every simplex of $D\left(Y_\tau,\tau,(R|_{\tau,Y_\tau})^T\right)$ is also a simplex of $D\left(Y_\sigma,\sigma,(R|_{\sigma,Y_\sigma})^T\right)$ whenever $\sigma \subseteq \tau$.
\end{proof}

These two Lemmas together imply that Theorem \ref{thm:cosheaf_r0_functor} extends immediately to a functoriality result for $\cshf{R}$.

\begin{corollary}
  \label{cor:cosheaf_r_functor}
  The transformation $(X,Y,R) \mapsto \cshf{R}$ is a covariant functor $CoShvRep: {\bf Rel} \to {\bf CoShvAsc}$.  If we restrict to ${\bf Rel_+} \to {\bf CoShvAsc}$, this becomes a faithful covariant functor.
\end{corollary}

\begin{theorem}
  \label{thm:cosheaf_global_cosections}
  The space of global cosections of $\cshf{R}$ is simplicially isomorphic to $D(Y,X,R^T)$, the Dowker complex for the transpose.
\end{theorem}
\begin{proof}
  Before we begin, notice that the vertices of $D(Y,X,R^T)$ are elements of $Y$ that are related to at least one element of $X$.  These are also elements of the costalks of $\cshf{R}$, and since the extensions of $\cshf{R}$ are inclusions, we need not worry about conflicting names for elements of $Y$.  Therefore, to establish this result, we simply need to show that every simplex $\sigma \in D(Y,X,R^T)$ appears in at least one costalk of $\cshf{R}$, and conversely that every simplex in every costalk of $\cshf{R}$ is also a simplex of $D(Y,X,R^T)$.

  Suppose that $\sigma=[y_0,y_1,\dotsc,y_n]$ is a simplex of $D(Y,X,R^T)$.  This means that there is an $x \in X$ such that $(x,y_i) \in R$ for all $i=0,\dotsc,n$.  Put another way, every $y_i \in \sigma$ is also an element of $Y_{[x]}$.  Therefore, the costalk $\cshf{R}([x])$ contains $\sigma$.

  Suppose that $\sigma=[y_0,y_1,\dotsc,y_n]$ is a simplex of $\cshf{R}(\tau)$ for some simplex $\tau$ of $D(X,Y,R)$.  This means that $\sigma$ is a simplex of $D\left(Y_\tau,\tau,(R|_{\tau,Y_\tau})^T\right)$, by definition.  That means that if we select any $x \in \tau$, it follows that $(y_i,x) \in (R|_{\tau,Y_\tau})^T \subseteq R^T$.  Therefore, $\sigma$ is a simplex of $D(Y,X,R^T)$. 
\end{proof}

What this means is that we have the following functorial diagram
\begin{equation*}
  \xymatrix{
    && {\bf Asc} \ar[dr]^G \\
    {\bf Rel} \ar[urr]^D \ar[d]_{Transp} \ar[rr]_-{CoShvRep} && {\bf CoShvAsc} \ar[u]_{Base} \ar[d]^{\Gamma} & {\bf CW}\\
    {\bf Rel} \ar[rr]_{D} && {\bf Asc} \ar[ur]_G \\
    }
\end{equation*}
where ${\bf CW}$ is the category of CW complexes and homotopy classes of continuous maps, $G$ is the geometric realization of an abstract simplicial complex, $Base$ is the functor that forgets the costalks of a cosheaf (Corollary \ref{cor:base_functor}), and $\Gamma$ is the functor that constructs the space of global cosections from a cosheaf (Lemma \ref{lem:morphisms_cosections}).  Dowker duality asserts that the top and bottom paths in this diagram are equivalent up to homotopy.

For a cosheaf $\cshf{C}$ on an abstract simplicial complex $X$ of abstract simplicial complexes whose extensions are inclusions, let us define a new cosheaf $Dual(\cshf{C})$ on the space of global cosections of $\cshf{C}$.  Noting that the space of global cosections $\cshf{C}(X)$ is also an abstract simplicial complex, suppose $\sigma$ is a simplex of $\cshf{C}(X)$.  Define the costalk $\left(Dual(\cshf{C})\right)(\sigma)$ to be the simplicial complex formed by the union of every simplex $\alpha$ in $X$ whose costalk $\cshf{C}(\alpha)$ contains $\sigma$.  Abstractly, this is equivalent to a union
\begin{equation*}
  \left(Dual(\cshf{C})\right)(\sigma) =  \bigcup \{ \alpha \in X : \sigma \in \cshf{C}(\alpha)\},% = \lim_{\to} \{ \alpha \in X : \sigma \in \cshf{C}(\alpha)\},
\end{equation*}
which implies that $Dual(\cshf{C})$ is a well-defined cosheaf when the extensions are all chosen to be inclusions.

\begin{figure}
  \begin{center}
    \includegraphics[height=2.75in]{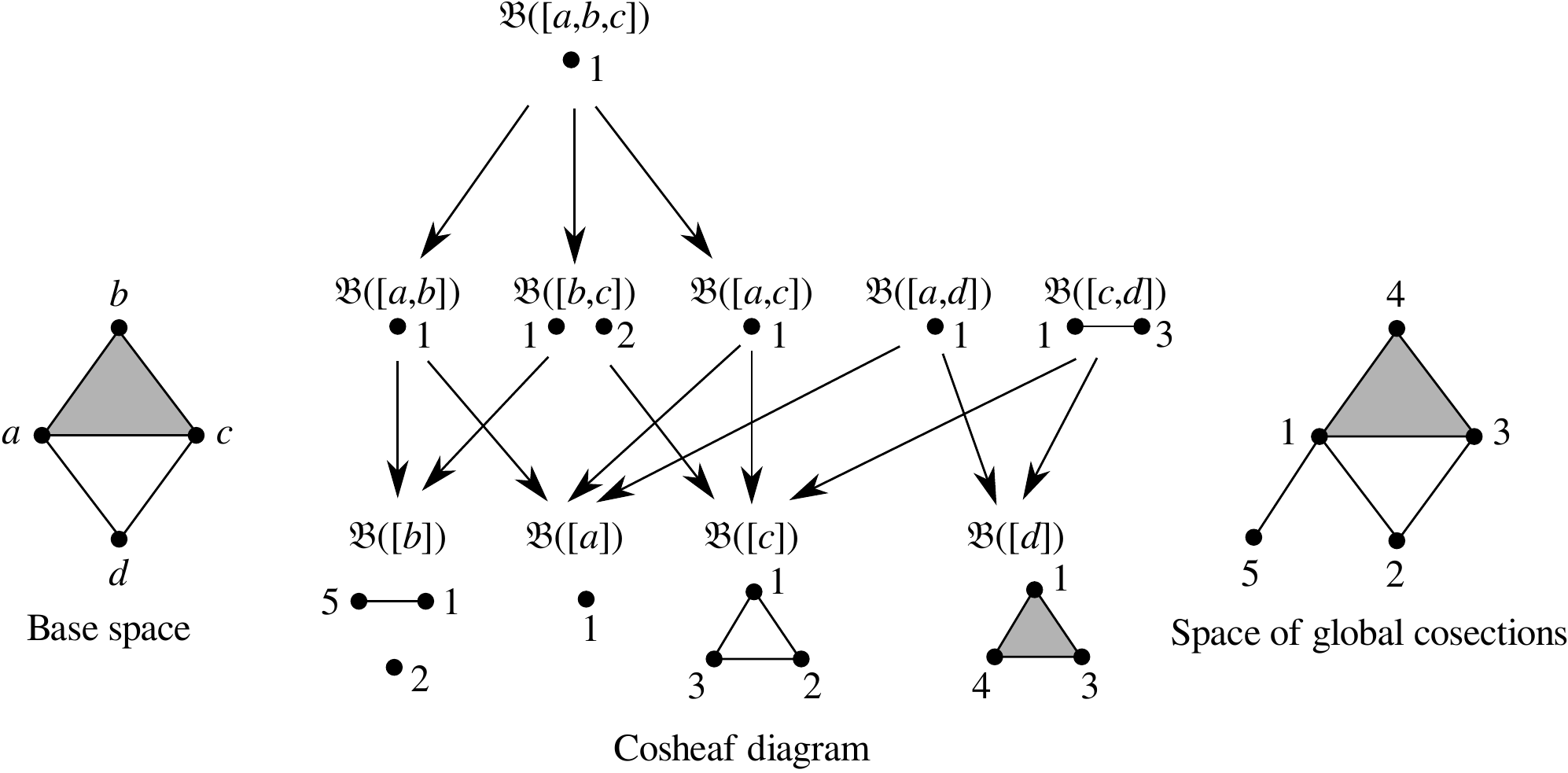}
    \caption{A cosheaf $\cshf{B}$ of abstract simplicial complexes described in Example \ref{eg:eg1_coshvasc}: (left) the base space of $\cshf{B}$, (center) the diagram of $\cshf{B}$, (right) the space of global cosections of $\cshf{B}$.}
    \label{fig:eg1_coshvasc}
  \end{center}
\end{figure}

\begin{figure}
  \begin{center}
    \includegraphics[height=2.5in]{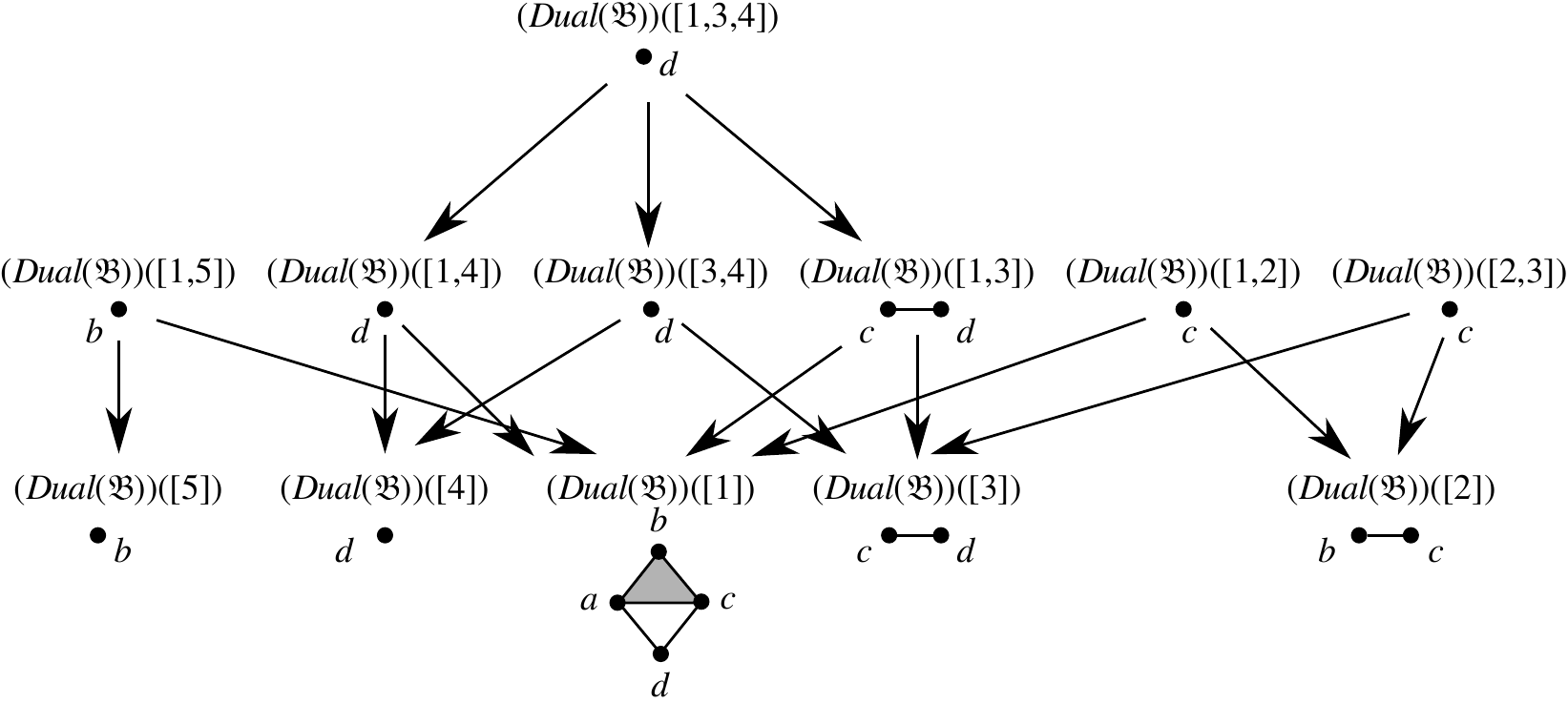}
    \caption{The cosheaf $Dual(\cshf{B})$ that is dual to the cosheaf $\cshf{B}$ shown in Figure \ref{fig:eg1_coshvasc} and described in Example \ref{eg:eg1_coshvasc}.}
    \label{fig:eg1_coshvasc_dual}
  \end{center}
\end{figure}

\begin{example}
  \label{eg:eg1_coshvasc}
  Figure \ref{fig:eg1_coshvasc} shows a cosheaf $\cshf{B}$ of abstract simplicial complexes on an abstract simplicial complex.  Since each extension map is an inclusion, this cosheaf is an object in ${\bf CoShvAsc}$.  The space of global cosections of this cosheaf is an abstract simplicial complex, which is shown at right in Figure \ref{fig:eg1_coshvasc}.  The cosheaf $Dual(\cshf{B})$ can therefore be constructed on this new abstract simplicial complex using the definition above.  The resulting cosheaf is shown in Figure \ref{fig:eg1_coshvasc_dual}, where it is clear that each extension of this new cosheaf is an inclusion.  It is also easily seen that the space of global cosections of $Dual(\cshf{B})$ is the base space of $\cshf{B}$.
\end{example}

\begin{lemma}
  $Dual$ is a covariant functor ${\bf CoShvAsc} \to {\bf CoShvAsc}$.
\end{lemma}
\begin{proof}
  The $Dual$ functor exchanges the base space with the space of global cosections.  For a cosheaf $\cshf{R}$ on $X$ that is an object of ${\bf CoShvAsc}$,
  \begin{equation*}
    Base(Dual(\cshf{R})) =\cshf{R}(X),
  \end{equation*}
  by definition and
  \begin{eqnarray*}
    \Gamma(Dual(\cshf{R}))&=& \left(\bigsqcup_{\sigma \in Base(Dual(\cshf{R}))} (Dual(\cshf{R}))(\sigma) \right)/\sim \\ % Definition of global cosections
    &=&\left(\bigsqcup_{\sigma \in \cshf{R}(X)} (Dual(\cshf{R}))(\sigma) \right)/\sim \\ % by above
    &=&\bigcup_{\sigma \in \cshf{R}(X)} (Dual(\cshf{R}))(\sigma) \\ % extensions are inclusions
    &=&\bigcup_{\sigma \in \cshf{R}(X)} \bigcup \{\alpha \in X : \sigma \in \cshf{R}(\alpha) \} \\ % Definition of stalks of Dual()
    &=&\left\{\alpha \in X : \text{there is a }\sigma \in \cshf{R}(X) \text{ such that }\sigma \in \cshf{R}(\alpha) \right\} \\ % Boolean algebra
    &=&\left\{\alpha \in X : \text{there is a }\sigma \in \bigcup_{\tau \in X} \cshf{R}(\tau) \text{ such that }\sigma \in \cshf{R}(\alpha) \right\} \\ % Definition of global cosections
    &=& X. % $\tau = \sigma$
  \end{eqnarray*}

  A cosheaf morphism $m: \cshf{R} \to \cshf{S}$ along a simplicial map $f: X \to Y$ induces a map $m_*: \cshf{R}(X) \to \cshf{S}(Y)$ on each space of cosections (Lemma \ref{lem:morphisms_cosections}).
  We use these data to define a morphism $w : Dual(\cshf{R}) \to Dual(\cshf{S})$.  
  As such, the induced map $m_*$ on the space of global cosections becomes the new base space map, along which the new cosheaf morphism $w$ is written.
  The individual simplices map by way of restricting to the components of the old morphism, since $m_*$ is a simplicial map.
  Conversely, the old base space map $f$ defines the new component maps $w_\sigma$ by restriction.

  Explicitly, if $\sigma$ is a simplex of $Base(Dual(\cshf{R})) = \cshf{R}(X)$, we have that
  \begin{equation*}
    (Dual(\cshf{R}))(\sigma) = \bigcup \{\alpha \in X : \sigma \in \cshf{R}(\alpha)\}.
  \end{equation*}
  The component $w_\sigma$ must be a function $(Dual(\cshf{R}))(\sigma) \to Dual(\cshf{S})(m_*(\sigma))$.
  Since every element of $(Dual(\cshf{R}))(\sigma)$ is an $\alpha \in X$, and the domain of $f$ is $X$, we may define
  \begin{equation*}
    w_\sigma = f|_{(Dual(\cshf{R}))(\sigma)}.
  \end{equation*}
  This is well-defined because if $\sigma \in \cshf{R}(\alpha)$ then $m_*(\sigma) \in \cshf{S}(f(\alpha))$, and
  \begin{equation*}
    (Dual(\cshf{S}))(m_*(\sigma)) = \bigcup \{\beta \in Y : m_*(\sigma) \in \cshf{S}(\beta)\}.
  \end{equation*}

  To establish that these component maps form a cosheaf morphism, we need to establish that the diagram below commutes for all simplices $\alpha \subseteq \beta$ in $\cshf{R}(X)$:
  \begin{equation*}
    \xymatrix{
      (Dual(\cshf{R}))(\beta) \ar[d]_{(Dual(\cshf{R}))(\alpha \subseteq \beta)} \ar[r]^-{w_\beta}& (Dual(\cshf{S}))(m_*(\beta))\ar[d]^{(Dual(\cshf{S}))(m_*(\alpha) \subseteq m_*(\beta))}\\
      (Dual(\cshf{R}))(\alpha) \ar[r]_-{w_\alpha}& (Dual(\cshf{S}))(m_*(\alpha)) \\      
      }
  \end{equation*}
  This follows because the vertical maps are inclusions and the horizontal maps are both restrictions of $f$ to nested subsets.

  Finally, composition of morphisms is preserved because that is simply composition of the base and global cosection functions.
\end{proof}

\begin{theorem} (Cosheaf version of Dowker duality)
  \label{thm:cosheaf_dowker_duality}
  $Dual$ is a functor that makes the diagram of functors commute:
  \begin{equation*}
    \xymatrix{
      {\bf Rel}  \ar[d]_-{Transp} \ar[rr]^-{CoShvRep} && {\bf CoShvAsc} \ar[d]^{Dual} \\
      {\bf Rel}\ar[rr]_-{CoShvRep} && {\bf CoShvAsc} \\
      }
  \end{equation*}
\end{theorem}

\begin{proof}
  The way that $Dual(\cshf{R})$ has been defined, we might end up with a simplicial complex as a costalk that is not a complete simplex -- which is a problem according to Lemma \ref{lem:costalks_complete_simplex} -- but this does not happen in the image of $CoShvRep$.  Suppose that $\cshf{R} = CoShvRep(X,Y,R)$.  We claim that for every simplex $\sigma$ in $\cshf{R}(X) = D(Y,X,R^T)$, the set of simplices
  \begin{equation*}
    \{ \alpha \in X : \sigma \in \cshf{R}(\alpha)\}
  \end{equation*}
  has a unique maximal simplex in the inclusion order, so that the union in the definition of $\left(Dual(\cshf{R})\right)(\sigma)$ really is just that one simplex.  To see that, suppose that $\alpha$ and $\beta$ are simplices of $X$ for which $\sigma \in \cshf{R}(\alpha)$ and $\sigma \in \cshf{R}(\beta)$.  Suppose that any other simplex $\gamma$ that contains $\alpha$ has $\sigma \notin \cshf{R}(\gamma)$, so $\alpha$ is maximal in the sense of inclusion.  We want to show that $\beta \subseteq \alpha$.  Going back to the definition of $\cshf{R}$, we have that
  \begin{equation*}
    \cshf{R}(\alpha) = D\left(Y_\alpha,\alpha,(R|_{\alpha,Y_\alpha})^T\right)
  \end{equation*}
    and
  \begin{equation*}
    \cshf{R}(\beta) = D\left(Y_\beta,\beta,(R|_{\beta,Y_\beta})^T\right).
  \end{equation*}
  Both contain $\sigma$.  What about the simplex $\delta$ whose vertices are the union of the vertices of $\alpha$ and $\beta$?  Suppose that $y\in Y$ is a vertex of $\sigma$.  This means that $y \in Y_\alpha \cap Y_\beta$, which means that $(x,y) \in R$ for every $x \in \alpha \cup \beta = \delta$.  Thus $\sigma \subseteq Y_\delta$, or in other words $\sigma \in \cshf{R}(\delta)$ according to Lemma \ref{lem:costalks_complete_simplex}.  On the other hand, if $\alpha \subsetneqq \delta$, we assumed that $\sigma \notin \cshf{R}(\delta)$.  So the only way this can happen is if $\alpha = \delta$, which implies $\beta \subseteq \alpha$.

  With this fact in hand, we can observe that
  \begin{eqnarray*}
    \left((Dual \circ CoShvRep)(X,Y,R)\right)(\sigma) &=& \left(Dual(\cshf{R})\right)(\sigma) \\
    &=& \bigcup \{ \alpha \in X : \sigma \in \cshf{R}(\alpha)\} \\
    &=& \bigcup \{\alpha \in X : \sigma \subseteq Y_\alpha\} \\
    &=& \bigcup \{\alpha \in X : \text{for all }y \in \sigma\text{ and all }x \in \alpha, \; (x,y) \in R\}\\
    &=& \{x \in X : (x,y)\text{ for all }y \in \sigma\} \\
    &=& D\left(X_\sigma,\sigma,R|_{X_\sigma,\sigma}\right) \\
    &=& \left(CoShvRep(Y,X,R^T)\right)(\sigma)\\
    &=& \left((CoShvRep \circ Transp)(X,Y,R)\right)(\sigma).
  \end{eqnarray*}
\end{proof}

\begin{figure}
  \begin{center}
    \includegraphics[height=2.75in]{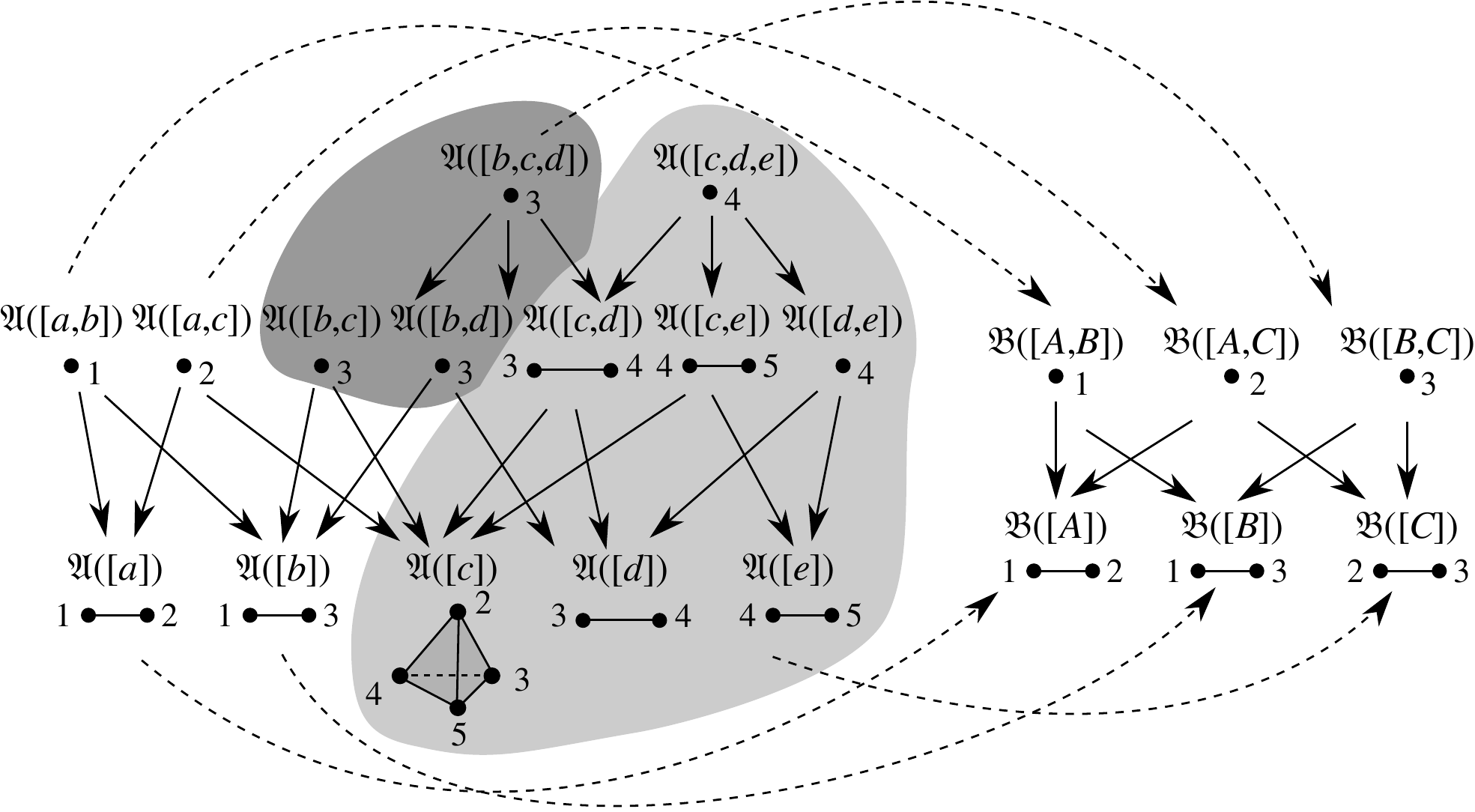}
    \caption{The cosheaf morphism induced by $CoShvRep$ described in Example \ref{eg:rel_cosheaf_dual_morph}.}
    \label{fig:rel_cosheaf_dual_morph}
  \end{center}
\end{figure}

\begin{figure}
  \begin{center}
    \includegraphics[height=2.5in]{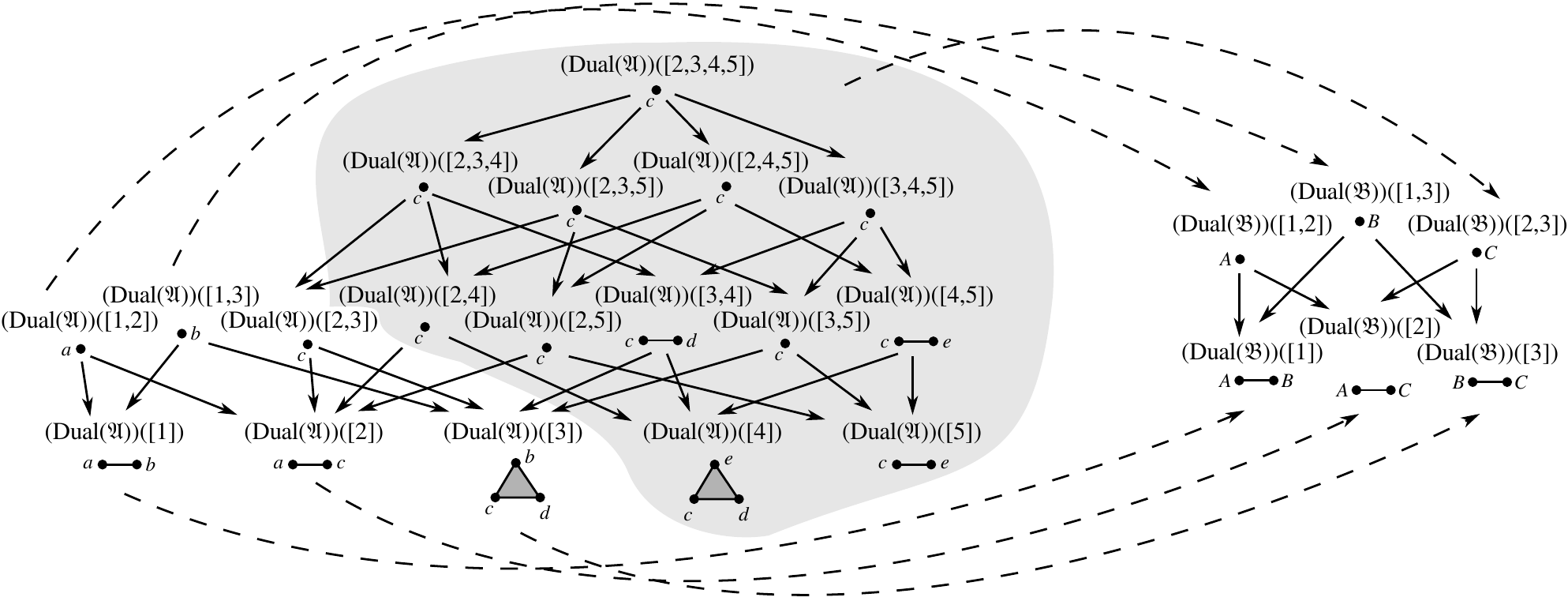}
    \caption{The cosheaf morphism induced by $Dual \circ CoShvRep$ described in Example \ref{eg:rel_cosheaf_dual_morph}.}
    \label{fig:rel_cosheaf_dual_morph_dual}
  \end{center}
\end{figure}

\begin{example}
  \label{eg:rel_cosheaf_dual_morph}
  Consider the relation morphism ${\bf Rel}$ morphism $(f,g) : (X_1,Y_1,R_1) \to (X_4,Y_4,R_4)$ defined in Example \ref{eg:rel_cosheaf_morph}.  Recall that the relation $R_1$ between the sets $X_1=\{a,b,c,d,e\}$ and $Y_1=\{1,2,3,4,5\}$, was given by the matrix
  \begin{equation*}
    r_1 = \begin{pmatrix}
      1&1&0&0&0\\
      1&0&1&0&0\\
      0&1&1&1&1\\
      0&0&1&1&0\\
      0&0&0&1&1\\
      \end{pmatrix},
  \end{equation*}
  and the relation $R_4$ between $X_4=\{A,B,C\}$ and $Y_4 =\{1,2,3\}$ was given by
  \begin{equation*}
    r_4 = \begin{pmatrix}
      1&1&0\\
      1&0&1\\
      0&1&1\\
    \end{pmatrix}.
  \end{equation*}
  The function $f: X_1 \to X_4$ was given by
  \begin{equation*}
    f(a) = A,\;     f(b) = B,\;     f(c) = C,\;     f(d) = C,\;     f(e) = C,
  \end{equation*}
  and the function $g: Y_1 \to Y_2$ was given by
  \begin{equation*}
    g(1) = 1,\;    g(2) = 2,\;    g(3) = 3,\;    g(4) = 3,\;    g(5) = 3.
  \end{equation*}
  Let is define $\cshf{A}=CoShvRep(X_1,Y_1,R_1)$ and $\cshf{B}=CoShvRep(X_4,Y_4,R_4)$.  The cosheaf morphism induced by $CoShvRep^0$ was shown in Figure \ref{fig:dowker_cosheaf_map}, but what interests us now is the cosheaf morphism $\cshf{A} \to \cshf{B}$ induced by $CoShvRep$ and $Dual(\cshf{A}) \to Dual(\cshf{B})$ induced by $Dual \circ CoShvRep$ (or equally well, induced by $CoShvRep \circ Transp$).  These two morphisms are shown in Figures \ref{fig:rel_cosheaf_dual_morph} and \ref{fig:rel_cosheaf_dual_morph_dual}, respectively.  Notice in particular that each component map (in both morphisms) is a simplicial map, so that whenever two vertices are collapsed (for instance $g(3)=g(4)=g(5)$) the associated simplices are collapsed as well.
\end{example}

\section{Redundancy of relations}

As has been explained earlier, the functor $D:{\bf Rel}\to{\bf Asc}$ is not faithful; many distinct relations have the same Dowker complex.  One way this can happen is if columns (or rows) in the matrix for the relation are \emph{redundant}, which means that a column (or row) has $1$s in all the same places as another column (or row), since this simply adds additional copies of the same simplex to the Dowker complex or its dual.  We can construct a (non-functorial) cosheaf to detect this redundancy directly using a similar construction to our earlier ones.

Since Lemma \ref{lem:costalks_complete_simplex} establishes that $D\left(Y_\sigma,\sigma,(R|_{\sigma,Y_\sigma})^T\right)$ is always a complete simplex for a relation $(X,Y,R)$ -- equivalently, the matrix for $R|_{\sigma,Y_\sigma}$ is a block of all $1$s -- it does not have any useful information beyond the vertex set $Y_\sigma$.  In a sense, $\cshf{R}^0=CoShvRep^0(X,Y,R)$ and $\cshf{R}=CoShvRep(X,Y,R)$ are basically very similar; the only difference being the topology on their costalks.

Taking a different perspective, the matrix for $R|_{X\backslash\sigma,Y_\sigma}$ contains all the information about the elements of $Y_\sigma$ that is not a result of their relation to elements in $\sigma$.  This observation means that we can also define a rather different cosheaf by its costalks
  \begin{equation*}
    \cshf{S}(\sigma) = D\left(Y_\sigma,X\backslash\sigma,(R|_{X\backslash\sigma,Y_\sigma})^T\right),
  \end{equation*}
  on each simplex $\sigma$ of $D(X,Y,R)$.  As in the previous constructions, we may take the extensions $\cshf{S}(\sigma \subseteq \tau)$ to be simplicial maps induced by inclusions.  The extensions are well defined because of Lemma \ref{lem:y_inclusion}.  If $[y_0,y_1, \dotsc, y_n]$ is a simplex of $D\left(Y_\tau,X\backslash\tau,(R|_{X\backslash\tau,Y_\tau})^T\right)$, then this means that there is an $x\notin \tau$ such that $(x,y_i) \in R$ for all $i=0,1,\dotsc, n$.  Evidently, $x \notin \sigma$ as well, so $[y_0,y_1, \dotsc, y_n]$ is a simplex of $D\left(Y_\sigma,X\backslash\sigma,(R|_{X\backslash\sigma,Y_\sigma})^T\right)$ as well.

  Elements of $Y_\sigma$ may not be related to any elements of $X$ besides those already in $\sigma$.  This means that $\cshf{R}^0=CoShvRep^0(X,Y,R)$ may not be a sub-cosheaf of $\cshf{S}$, because a costalk of $\cshf{S}$ may have fewer vertices than in the corresponding costalk of $\cshf{R}^0$.  Moreover, the transformation $(X,Y,R) \mapsto \cshf{S}$ is not functorial.
  
  \begin{example}
    \label{eg:redundant_nonfunctor}
    Let $(X,Y,R)$ be any relation, and let $(X',Y',R')$ be the relation given by $(\{a,b\},\{1,2\}, \{(a,1),(b,2)\})$.  If $f:X \to X'$ is the constant function that takes the value $f(x)=a$ on all $x\in X$, then the constant function $g: Y \to Y'$ with $g(y) = 1$ for all $y\in Y$ will define a ${\bf Rel}$ morphism $(f,g) : (X,Y,R) \to (X',Y',R')$.

    Notice that the cosheaf $\cshf{T}$ constructed by the recipe above on $(X',Y',R')$ has
    \begin{equation*}
      \cshf{T}([a]) =  D\left(Y_{[a]},X\backslash [a],(R|_{X\backslash [a],Y_{[a]}})^T\right) = D\left(\{1\},\{b\},\emptyset\right) = \emptyset,
    \end{equation*}
    but
    \begin{equation*}
      \cshf{S}(\sigma) = D\left(Y_\sigma,X\backslash\sigma,(R|_{X\backslash\sigma,Y_\sigma})^T\right)
    \end{equation*}
    may well be nonempty.  Since $f(\sigma) = [a]$ by construction, this means that there is no way to construct a cosheaf morphism $\cshf{S} \to \cshf{T}$ along $D(f):D(X,Y,R) \to D(X',Y',R')$, since the component $m_\sigma$ would be a function
    \begin{equation*}
      m_\sigma: D\left(Y_\sigma,X\backslash\sigma,(R|_{X\backslash\sigma,Y_\sigma})^T\right) \to \emptyset,
    \end{equation*}
    which cannot exist unless the domain is empty.

    Finally, notice that restricting the domain to ${\bf Rel_+}$ does not improve matters, since $(X',Y',R')$ is already an object of ${\bf Rel_+}$.
  \end{example}

  Even in the face of situations like Example \ref{eg:redundant_nonfunctor}, sometimes a cosheaf morphism can be induced by a ${\bf Rel}$ morphism.

  \begin{figure}
  \begin{center}
    \includegraphics[height=2.75in]{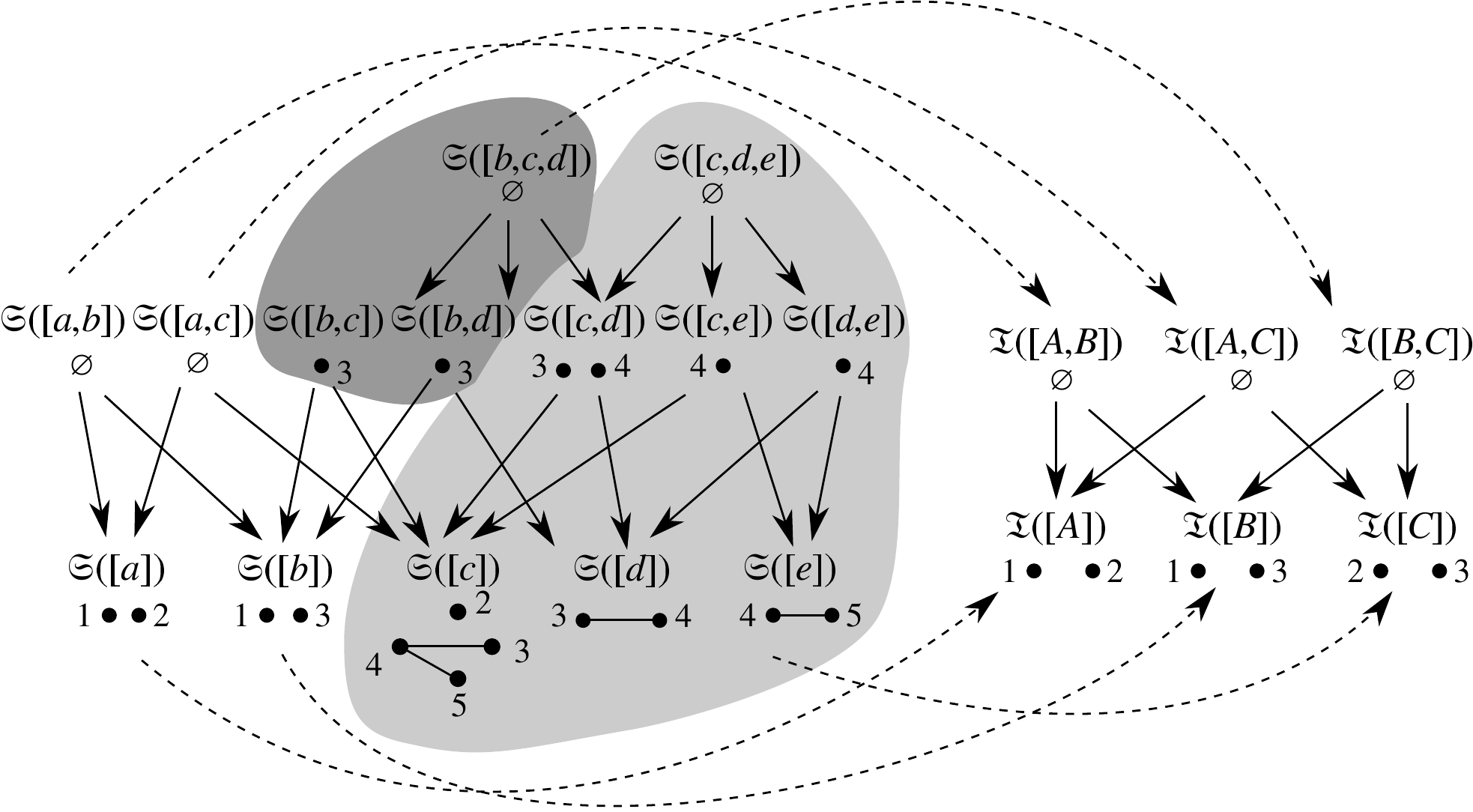}
    \caption{The cosheaf morphism described in Example \ref{eg:rel_cosheaf2_morph}.  (Compare with Figures \ref{fig:dowker_cosheaf_map} and \ref{fig:rel_cosheaf_dual_morph}, which are induced by the same ${\bf Rel}$ morphism.)} 
    \label{fig:rel_cosheaf2_morph}
  \end{center}
\end{figure}

  \begin{example}
    \label{eg:rel_cosheaf2_morph}
    Consider the ${\bf Rel}$ morphism $(f,g) : (X_1,Y_1,R_1) \to (X_4,Y_4,R_4)$ given in Example \ref{eg:rel_cosheaf_morph}.  If we construct $\cshf{S}$ from $(X_1,Y_1,R_1)$ and $\cshf{T}$ from $(X_4,Y_4,R_4)$ using the recipe above, this happens to induce a cosheaf morphism $\cshf{S} \to \cshf{T}$, which is shown in Figure \ref{fig:rel_cosheaf2_morph}.  It is immediately apparent that the costalks on all of the maximal simplices are empty.  This is a consequence of the definition: if $\sigma$ is a maximal simplex of $D(X,Y,R)$, then this means that for any $y \in Y_\sigma$, $y \notin Y_\tau$ for any strictly larger $\tau$ that contains $\sigma$.  The presence of $\emptyset$ in various costalks is not a problem for the morphism, since there always exists a unique function $\emptyset \to A$ for any set $A$.  Furthermore, even though in Example \ref{eg:redundant_nonfunctor} the presence of empty sets in the codomain caused a problem, they are benign in this case because the domain costalks are also empty.

    A little inspection reveals that the costalks identify redundant simplices in $D(Y,X,R^T)$.  Such a redundant simplex is generated by a row of the matrix for $R$ that is a proper subset of some other row.  This means that we can interpret the space of global cosections of $\cshf{S}$ (or $\cshf{T}$) as being the collection of all redundant simplices -- those whose corresponding elements of $X$ can be removed without changing the Dowker complex.
  \end{example}

\section*{Acknowledgments}
This material is based upon work supported by the Defense Advanced Research Projects Agency (DARPA) SafeDocs program under contract HR001119C0072.  Any opinions, findings and conclusions or recommendations expressed in this material are those of the author and do not necessarily reflect the views of DARPA.
 
\bibliographystyle{plainnat}
\bibliography{relations_bib}
\end{document}